\documentclass[a4paper]{amsart}

\usepackage{amssymb}
\usepackage{latexsym}

\usepackage{amsmath,amsfonts}

\usepackage{placeins}
  
\usepackage{enumerate}

\usepackage{xifthen}
\usepackage{easyvector}
\usepackage{epigraph}
\usepackage{algorithm}

\usepackage{adjustbox}

\usepackage{pgfplots}
\pgfplotsset{compat=newest}
\usetikzlibrary{arrows}
\usetikzlibrary{calc}
\usetikzlibrary{shapes}

\usetikzlibrary{fit,arrows}
\usetikzlibrary{positioning}
\usetikzlibrary{automata}

\tikzset{
  font={\fontsize{12pt}{12}\selectfont}}

\newfont{\NUMBERS}{msbm8 scaled\magstep1}

\newcommand{\REAL}{\mbox{\NUMBERS R}}

\usepackage{ifthen}

\newcommand{\Vect}[2][]
{
  \ifthenelse{\equal{#1}{}}
  {\boldsymbol{#2}}
  {{#2}_{#1}}
}
\newcommand{\Matr}[2][]
{
  \ifthenelse{\equal{#1}{}}
  {\boldsymbol{#2}}
  {{#2}_{#1}}
}

\newcommand{\Argmin}{\operatornamewithlimits{argmin\vphantom{q}}}

\newcommand{\Var}{\ensuremath{\operatorname{var}}}
\newcommand{\Err}{\ensuremath{\operatorname{err}}}

\newcommand{\Tendsto}{\rightarrow}

\newcommand{\dx}{\, dx}

\newcommand{\Haus}[1]{\mathcal{H}^{#1}}
\DeclareMathOperator{\Div}{\operatorname{\nabla\cdot}}
\DeclareMathOperator{\Grad}{\nabla}
\DeclareMathOperator{\Sign}{\operatorname{sign}}

\def\XXint#1#2#3{{\setbox0=\hbox{$#1{#2#3}{\int}$ }
\vcenter{\hbox{$#2#3$ }}\kern-.6\wd0}}

\newcommand{\ABS}[1]{\left|#1\right|}
\newcommand{\NORM}[1]{\left\|#1\right\|}
\newcommand{\Dt}[1]{#1'}

\newcommand{\Lyap}{\ensuremath{\mathcal{L}}}

\newcommand{\Wmass}{\ensuremath{\mathcal{M}}}
\newcommand{\Ene}{\ensuremath{\mathcal{E}_{\Forcing}}}

\newcommand{\Dim}{d}
\newcommand{\Domain}{\ensuremath{\Omega}}

\newcommand{\tstep}{k}
\newcommand{\tstepp}{k+1}
\newcommand{\Deltat}[1][]{
  \ifthenelse{\equal{#1}{}}
  {\Delta t_{\tstep}}
  {\Delta t_{#1}}
  }

\newcommand{\PolySymb}{\mathcal{P}}
\newcommand{\PC}[1]{\ensuremath{\PolySymb_{#1}}}
\newcommand{\PONE}{\ensuremath{\PolySymb_{1}}}

\newcommand{\PZERO}{\ensuremath{\PolySymb_{0}}}
\newcommand{\MeshPar}{h}
\newcommand{\Cell}[1][]{\ifthenelse{\equal{#1}{}}{T}{T_{#1}}}
\newcommand{\Tsymb}{\mathcal{T}}
\newcommand{\Triang}[1][]
{
  \ifthenelse{\equal{#1}{}}
  {\Tsymb}
  {\Tsymb_{#1}}
}
\newcommand{\PotOp}{\Pot}
\newcommand{\PotOf}[2][]
{
  \ifthenelse{\equal{#1}{}}
  {\Pot_{\Forcing}(#2)}
  {\Pot_{#1}(#2)}
}

\newcommand{\Vof}[2]{[#1]^{#2}}

\newcommand{\Radon}{\mathcal{M}}

\newcommand{\Lspacechar}{L}
\newcommand{\Lspace}[2][]{
  \ifthenelse{\equal{#1}{}}
  {\Lspacechar^{#2}}
  {\Lspacechar_{#1}^{#2}}
}
\newcommand{\Lplus}[1]{\Lspace{#1}_{+}}

\newcommand{\Cachar}{\mathcal{C}}

\newcommand{\Cont}[1][]{
  \ifthenelse{\equal{#1}{}}
  {\Cachar} 
  {\Cachar^{#1}}
}

\newcommand{\HolderExp}{\delta}
\newcommand{\Holder}[1][]{
  \ifthenelse{\equal{#1}{}}
  {\Cachar^{\HolderExp}}
  {\Cachar^{#1}}  
}

\newcommand{\Lip}[1]
{
  \ifthenelse{\equal{#1}{}}
  {\text{Lip}}
  {\text{Lip}_{#1}}        
}

\newcommand{\Sob}[2][]{
  \ifthenelse{\equal{#2}{}}
  {H^{#1}}
  {W^{#1,#2}}  
}

\newcommand{\DHolder}[2][]
{
  \ifthenelse{\equal{#2}{}}
  {\Cachar^{#1,\HolderExp}}   
  {\Cachar^{#1,#2}}  
}

\newcommand{\VspaceSymb}{\mathcal{V}}
\newcommand{\Vspace}[1][]{
  \ifthenelse{\equal{#1}{}}
  {\VspaceSymb_{\MeshPar}}
  {\VspaceSymb_{\MeshPar}(#1)}
}
\newcommand{\Vdim}{N}
\newcommand{\VbaseSymb}{\varphi}
\newcommand{\Vbase}[1][]{
  \ifthenelse{\equal{#1}{}}
  {\VbaseSymb}
  {\VbaseSymb_{#1}}
}
\newcommand{\WspaceSymb}{\mathcal{W}}
\newcommand{\Wspace}[1][]{
  \ifthenelse{\equal{#1}{}}
  {\WspaceSymb_{\MeshPar}}
  {\WspaceSymb_{\MeshPar}(#1)}
}
\newcommand{\Wdim}{M}
\newcommand{\WbaseSymb}{\psi}
\newcommand{\Wbase}[1][]{
  \ifthenelse{\equal{#1}{}}
  {\WbaseSymb}
  {\WbaseSymb_{#1}}
}
\newcommand{\Tdens}{\mu}
\newcommand{\TdensIni}{\Tdens_0}

\newcommand{\TdensH}{\Tdens_{\MeshPar}}
\newcommand{\OptTdensH}{\TdensH^*}
\newcommand{\OptTdens}{\Tdens^*}
\newcommand{\Pot}{u}
\newcommand{\PotH}{\Pot_{\MeshPar}}
\newcommand{\OptPot}{\Pot^*}
\newcommand{\OptPotH}{\PotH^*}
\newcommand{\Dirac}[1]{\delta_{#1}}

\newcommand{\Equi}[1]{#1^{*}}
\newcommand{\Vel}{v}
\newcommand{\OptVel}{\Vel^*}
\newcommand{\Field}{w}

\newcommand{\Ftest}{\varphi}

\newcommand{\TdVec}[1][]{
  \ifthenelse{\equal{#1}{}}     
  {\Vect{\Tdens}}                
  {\Vect{\Tdens}^{#1}}           
  }
\newcommand{\UVec}[1][]{
  \ifthenelse{\equal{#1}{}}     
  {\Vect{\Pot}}                  
  {\Vect{\Pot}^{#1}}             
  }                             
\newcommand{\TdensHIni}[1][]{
  \ifthenelse{\equal{#1}{}}
  {\TdVec[0]}
  {\TdVec[0]_{#1}}
}

\newcommand{\Potp}{u_p}

\newcommand{\PLaplSpace}{\Sob[1]{p}(\Domain)}

\newcommand{\Pflux}{\beta}

\newcommand{\Pbranch}{\PPP}

\newcommand{\Plapl}{p}

\newcommand{\Forcing}{f}

\newcommand{\Source}{\Forcing^{+}}
\newcommand{\Sink}{\Forcing^{-}}

\newcommand{\Fradial}{F}

\newcommand{\TolTime}{\tau_{\mbox{{\scriptsize T}}}}

\newcommand{\PPP}{q}

\newcommand{\Eps}{\varepsilon}
\newcommand{\Pwsymb}{r}
\newcommand{\Pwmass}[1][]
{
  \ifthenelse{\equal{#1}{}}
  {\Pwsymb}
  {\Pwsymb(#1)}
}

\newcommand{\BallSymb}{B}
\newcommand{\Ball}[2][]
{
  \ifthenelse{\equal{#2}{}}
  {\BallSymb_{#1}}
  {\BallSymb(#2,#1)}
}
\newcommand{\Avgint}[3][]
{
  \ifthenelse{\equal{#1}{}}
  {({#3})_{#2}}
  {(#3)_{#1,#2}}
}

\newcommand{\LCF}{Lyapunov-candidate functional}

\newcommand{\MK}{MK}

\newcommand{\MKEQS}{MK equations}
\newcommand{\OT}{OT}

\newcommand{\OTD}{OT\ density}

\newcommand{\BT}{BT}
\newcommand{\BTP}{BTP}
\newcommand{\CTP}{CTP}
\newcommand{\DMK}{DMK}

\newcommand{\PCG}{PCG}
\newcommand{\SPD}{SPD}

\usepackage{cleveref}

\newtheorem{Prop}{Proposition}
\newtheorem{Conject}{Conjecture}

\newtheorem{Remark}{Remark}

\crefname{equation}{Eq.}{Eqs.}
\crefname{theorem}{Theorem}{Theorems}
\crefname{chapter}{Chapter}{Chapters}
\crefname{figure}{Figure}{Figures}
\crefname{Conject}{Conjecture}{Conjectures}
\crefname{Problem}{Problem}{Problems}
\crefname{Prop}{Proposition}{Propositions}        
\crefname{Theo}{Theorem}{Theorems}
\crefname{Lemma}{Lemma}{Lemma}
\crefname{Corollary}{Corollary}{Corollaries}
\crefname{section}{Section}{Sections}

\usepackage{graphicx}

\usepackage[numbers,sort&compress]{natbib}

\usepackage{soul}
\usepackage{xcolor}
\usepackage{cancel}
\usepackage[normalem]{ulem}
\newcommand{\added}[1]{#1}
\newcommand{\replaced}[2]{#1}
\newcommand{\deleted}[1]{}
\newcommand{\sep}{,\ }

\begin{document}

\markboth{E. Facca et al.}{Branching structures emerging from a
  continuous optimal transport model}

\title{Branching structures emerging from a continuous optimal transport
  model
}

\author{ENRICO FACCA, FRANCO CARDIN, MARIO PUTTI}
\address{Centro Ennio de Giorgi,
  Scuola Normale Superiore,
  Piazza dei Cavalieri 7, Pisa, Italy enrico.facca@sns.it
  \\
  Deparment of Mathematics Tullio Levi Civita, University of Padua,\\
  Via Trieste 62, Padova, Italy,
  \{cardin,putti\}@math.unipd.it
}

\maketitle

\begin{abstract}
  Recently a Dynamic-Monge-Kantorovich formulation of the PDE-based
  $\Lspace{1}$-optimal transport problem was presented. The model
  considers a diffusion equation enforcing the balance of the
  transported masses with a time-varying conductivity that evolves
  proportionally to the transported flux.
  In this paper we present an extension of this model that considers a
  time derivative of the conductivity that grows as a power law of the
  transport flux with exponent $\Pflux>0$. A sub-linear growth
  ($0<\Pflux<1$) penalizes the flux intensity and promotes distributed
  transport, with equilibrium solutions that are reminiscent of
  Congested Transport Problems.
  On the contrary, a super-linear growth ($\Pflux>1$) favors flux
  intensity and promotes concentrated transport, leading to the
  emergence of steady-state ``singular'' and ``fractal-like''
  configurations that resemble those of Branched Transport Problems.
  We derive a numerical discretization of the proposed model that is
  accurate, efficient, and robust for a wide range of scenarios. For
  $\Pflux>1$ the numerical model is able to reproduce highly irregular
  and fractal-like formations without any a-priory structural
  assumption.
\end{abstract}

\keywords{
  Optimal Transport Problems \sep
  Congested, Branched, Ramified Transport\sep
  $\Plapl$-Laplacian\sep
  Monge-Kantorovich formulation\sep
  Numerical solution
}

\subjclass[2000]{
  49K20 \sep 
  49M25 \sep
  49M29 \sep
  35J70 \sep
  65N30 
}

\section{Introduction}
In this paper we propose and analyze theoretically and numerically an
extension of the Dynamic Monge-Kantorovich (\DMK) Optimal Transport
(\OT) model recently presented in~\citet{Facca-et-al:2018}.  The
\DMK\ can be summarized as follows.  Consider an open, bounded,
connected, and convex domain $\Omega$ in $\REAL^{\Dim}$ with smooth
boundary.  Given two non-negative rate densities, $\Source$ and
$\Sink$ such that $\int_{\Domain}\Source\dx=\int_{\Domain}\Sink\dx$,
\added{consider $\Forcing:=\Source-\Sink$. We want to} find the
pair $(\Tdens,\Pot):[0,+\infty[\times\Domain\mapsto
    \REAL^+\times\REAL$ that solves:
\begin{subequations}
  \label{eq:sys-1}
  \begin{align}
    \label{eq:sys-1-div}
    &  -\Div\left(\Tdens(t,x)\Grad\Pot(t,x)\right) 
      = \Forcing(x)\deleted{:=\Source(x)-\Sink(x)},
    \\ 
    \label{eq:sys-1-dyn}
    & \Dt{\Tdens}(t,x) =
      \Tdens(t,x)\ABS{\Grad\Pot(t,x)} - \Tdens(t,x) ,
    \\
    \label{eq:sys-1-bc} 	
    & \Tdens(0,x) =\TdensIni(x)> 0 ,
  \end{align}
\end{subequations}
completed with zero Neumann boundary conditions. Here $\Dt{\Tdens}$
denotes derivative with respect to time. \added{ Hereafter, we
  omit the dependence from $x$ and kept only the
  time-dependence. We will use this convention whenever no confusion
  arises.  }

In~\citet{Facca-et-al:2018} it is conjectured that the solution pair
$(\Tdens(t),\Pot(t))$ as function of time converges for
$t\Tendsto\infty$ towards $(\OptTdens,\OptPot)$ where, $\OptTdens$ is
the \OTD\ and $\OptPot$ is a Kantorovich Potential associated to
$\Source$ and $\Sink$. The pair $(\OptTdens,\OptPot)$ solves the
Monge-Kantorovich (\MK) partial differential equations introduced
in~\citet{Evans-Gangbo:1999}.
Local existence and uniqueness of $(\Tdens(t),\Pot(t))$ for
$t\in[0,\tau_{0}[$, with $\tau_0>0$ depending on the initial data, was
    proved under the hypotheses
    \deleted{$\TdensIni\in\Holder{}(\Domain)$} \added{ that
      $\TdensIni$ is H\"older continuous}
    and $\Forcing\in \Lspace{\infty}(\Domain)$. However, well
    posedness and full regularity of the solution, as well as its
    convergence towards the solution of the \MK\ equations, are still
    open issues.  The conjecture is strongly supported by convincing
    numerical results reported
    in~\citet{Facca-et-al:2018,Facca-et-al:sisc:2018}. Additionally, a
    \LCF\ for the dynamics is given by:
\begin{equation*}
  \Lyap(\Tdens)
  = 
  \frac{1}{2}
  \int_{\Domain}{
    \Tdens \ABS{\Grad \PotOf{\Tdens}}^2 
  }\dx 
  +
  \frac{1}{2}
  \int_{\Domain}{
    \Tdens
  }\dx,
\end{equation*}
where $\PotOf{\Tdens}$ indicates the weak solution of the elliptic
equation~\cref{eq:sys-1-div} for a given $\Forcing$. It is possible to
prove that $\Lyap$ decreases along the $\Tdens(t)$-trajectory and the
\OTD\ is its unique minimizer~\citep{Facca-et-al:sisc:2018}.

The extension we propose in this paper, suggested by the discrete
counterpart reported in~\citet{Tero-et-al:2007}, modifies the dynamics
of $\Tdens(t)$ by introducing a power law of the transport flux
$\Tdens(t)\ABS{\Grad\Pot(t)}$ with exponent $\Pflux>0$.
\deleted{Note that we have omitted the explicit dependence from
  $x$ and kept only the time-dependence. We will use this convention
whenever no confusion arises.  Using this notation } The proposed model
can be described as the problem of finding the pair
$(\Tdens(t),\Pot(t))$ such that:
\begin{subequations}
  \label{eq:sys-pflux}
  \begin{align}
    \label{eq:sys-pflux-div}
    &  -\Div\left(\Tdens(t)\Grad\Pot(t)\right) 
      = \Forcing\deleted{=\Source-\Sink} ,
    \\ 
    \label{eq:sys-pflux-dyn}
    & \Dt{\Tdens}(t) =
      \left[\Tdens(t)\ABS{\Grad\Pot(t)}\right]^{\Pflux}
      - \Tdens(t) ,
    \\
    \label{eq:sys-pflux-bc} 	
    & \Tdens(0) =\TdensIni(x)> 0 ,
  \end{align}
\end{subequations}
completed with zero Neumann boundary conditions.  Assuming
well-posedness of the above system, we claim that, as in the case
$\Pflux = 1$, the solution pair $(\Tdens(t),\Pot(t))$ converges toward
an equilibrium configuration $(\Equi{\Tdens}, \Equi{\Pot})$ as
$t\Tendsto\infty$. Moreover, we claim that this equilibrium point
should be related to congested ($0<\Pflux<1$) and branched
($\Pflux>1$) OT problems.  In fact, intuitively, a sub-linear growth
should penalize flux intensity (i.e. the transport density) and
promote distributed transport. Correspondingly, the equilibrium
solutions should be reminiscent of Congested Transport Problem (CTP),
\added{a branch of \OT\ theory that studies reallocation
  problems where mass concentration is penalized, and finds
  applications in, e.g., crowd motion and urban traffic (see, e.g.,
  \cite{Santambrogio:2015,BrascoPetrache:2014,Buttazzo-et-al:2009}).}
On the other hand, a super-linear growth should favor flux intensity
and promote concentrated transport, leading to the emergence of
“singular” and “fractal-like” configurations that resemble the
structures typical of Branched Transport Problem (BTP).  \added{
  Branched transport is an area of \OT\ that studies reallocation
  problems where mass concentration is encouraged along the transport
  paths, favoring ``economy of scale''. This is a common strategy that
  can be easily assumed to be a fundamental mechanism in the
  development of many natural systems such as, e.g., tree branches and
  roots, blood vessels, river networks, etc.~\citep{Banavar:1999,
    Banavar:2014, Rinaldo:2014}.  }
Note that any homogeneous positive function can replace the power law,
thus non-linearly modulating the mixing of the different behaviors. We
are interested in the simpler case of the power law as a model problem
that incorporates all the interesting responses.

The above claims are supported by extensive numerical experiments on a
number of different two-dimensional test cases and by some partial
results and heuristic justifications. Among the latter, we are
able to derive a \LCF\ $\Lyap_{\Pflux}$ given by:
\begin{subequations}
  \label{eq:lyap}
  \begin{gather}
    \label{eq:lyap-pflux}
    \Lyap_{\Pflux}(\Tdens) := \Ene(\Tdens) + \Wmass_{\Pflux} (\Tdens) ,
    \\
    \label{eq:ene-wmass-pflux}
    \Ene(\Tdens):= \frac{1}{2} 
    \int_{\Domain}{\Tdens\ABS{\Grad\PotOf{\Tdens}}^2}\dx ;
    \qquad
    \Wmass_{\Pflux}(\Tdens):=
    \begin{cases}
      \displaystyle
        \frac{1}{2}
        \int_{\Omega}\ln(\Tdens)\dx &\text{if } \Pflux=2
      \\
      \displaystyle
      \frac{1}{2}
      \int_{\Omega}
        \frac{\Tdens^{\frac{2-\Pflux}{\Pflux}}}{\frac{2-\Pflux}{\Pflux}}\dx
      &\text{otherwise}     
    \end{cases} .
  \end{gather}
\end{subequations}
 \replaced{ Similarly to the case $\Pflux=1$, $\Lyap_{\Pflux}$
   decreases in time along the $\Tdens(t)$-trajectory, thus it is
   natural to investigate if the \LCF\ admits a minimum, which is the
   natural limit candidate for our dynamics.  Intuitively, such minimum
   provides a trade-off between the transport cost, measured by the
   term $\Ene(\Tdens)$, and the ``infrastructure'' cost, measured by the
   term $\Wmass_{\Pflux}(\Tdens)$. For $0<\Pflux\leq 1 $ the second
   term is convex, penalizing the concentration of the support of
   $\Tdens$, as prescribed by the CTP. On the other hand, for
   $1<\Pflux<2$ $\Wmass_{\Pflux}(\Tdens)$ is concave, favoring the
   concentration of $\Tdens$ on narrow supports.  }{ Similarly to the
   case $\Pflux=1$, $\Lyap_{\Pflux}$ decreases in time along the
   $\Tdens(t)$-trajectory and, for certain values of $\Pflux$, we can
   give the exact characterization of the minimum of the \LCF, which
   is the natural limit candidate for our dynamics.  }

For the case $0<\Pflux<1$, we claim that the pair
$(\Tdens(t),\Pot(t))$ converges at long times toward
$(\ABS{\Grad\Potp}^{\Plapl-2},\Potp)$, with
$\Plapl=\frac{2-\Pflux}{1-\Pflux}$, where $\Potp$ is the solution of
the $\Plapl$-Poisson equation with forcing term
$\Forcing\deleted{=\Source-\Sink}$:
\begin{equation*}
  -\Div(\ABS{\Grad\Potp}^{\Plapl-2}\Grad\Potp)
  =\Forcing\deleted{=\Source-\Sink} .
\end{equation*}
This conjecture is supported by the fact that, for $\PPP=2-\Pflux$,
minimization of $\Lyap_{\Pflux}$ is equivalent to the following
variational problem:
\begin{equation}
  \label{eq:min-divergence}
  \inf_{\Vel\in\Vof{\Lspace{\PPP}(\Domain)}{\Dim}}
  \left\{
    \int_{\Domain}
    \frac{
      \ABS{\Vel}^{\PPP}
    }{
      \PPP
    }
    \dx
    \ : \ 
    \Div(\Vel)=\Forcing\deleted{=\Source-\Sink}
  \right\},
\end{equation}
that is a classical formulation of the CTP. \deleted{a branch
  of \OT\ theory that studies reallocation problems where mass
  concentration is penalized, and finds applications in, e.g., crowd
  motion and urban traffic
  }
The equivalence between problem~\cref{eq:min-divergence} for
$\PPP\in[1,2]$ and the solution of the $\Plapl$-Poisson equation, with
$\Plapl$ conjugate to $\PPP$, is a well known
result~\citep{Ekeland-Teman:1999}.  The equivalence for
$\Pflux=1,\PPP=1$ has been proved in~\citet{Facca-et-al:2018}.
Analogously to the case $\Pflux=1$ studied
in~\cite{Facca-et-al:2018,Facca-et-al:sisc:2018}, this new formulation
of the $\Plapl$-Poisson equation leads to robust and accurate
numerical discretization schemes, providing an unconventional yet very
efficient approach, at least with respect to the classical augmented
Lagrangian strategy usually adopted for the numerical solution of the
$\Plapl$-Poisson equation~\citep{Glowinski-Marrocco:1975,
  Barrett-Liu:1993, Benamou-et-al:2015}.

The case $\Pflux>1$ is mostly addressed by numerical experimentation.
A number of two-dimensional tests suggest a connection between the
steady state equilibrium solution $(\Equi{\Tdens}, \Equi{\Pot})$ of
the proposed model and solutions of the \replaced{BTP, in particular,
  our reference formulation is that one introduced
  by~\citet{Xia:2015}. In this formulation, ~\cref{eq:min-divergence}
  is rewritten in the case $0<\PPP<1$, where the integral must be
  reinterpreted appropriately. Thus we now look for a vector valued
  measure $\Vel$ solving for $0<\PPP<1$ the following minization problem
  \begin{equation}
    \label{eq:branch-intro}
    \inf_{\Vel\in\Vof{\Radon(\Domain)}{\Dim}}
    \left\{
    \int_{\Domain}
    \ABS{\frac{d\Vel}{d\Haus{1}}}^{\PPP}
      d\Haus{1}
    \ : \ 
    \Div(\Vel)=\Forcing\dx
    \right\},
  \end{equation}
  where $d\Haus{1}$ denote the Hausdorff measure of dimension 1 (see
  ~\citep{Santambrogio:2015, Xia:2003} for the detailed definition).
}{
  Branched Transport Problem (\BTP).  Branched transport is an area of
  \OT\ that studies reallocation problems where mass concentration is
  encouraged along the transport paths, favoring ``economy of
  scale''. This is a common strategy that can be easily assumed to be
  a fundamental mechanism in the development of many natural systems
  such as, e.g., tree branches and roots, blood vessels, river
  networks, etc.
  In
  particular, our reference \BTP\ formulation is the one introduced
  by
  that looks for minimizers
  of
  in the case $0<\PPP<1$.  Here, the
  vector field $\Vel$ is a vector valued measure and the integral must
  be reinterpreted appropriately.}  Although
in our case we are still not able to exactly identify the relations
with the reference \BTP\, formal calculations backed up by several
numerical results, suggest strong connections with the functionals
minimized in more classical \BT\ problems. More precisely, we are not
able to rigorously consider the singular measures typically arising in
\BT\ transport, but the numerical results are convincingly leading to
structures that closely resemble the \BT\ transport solutions.

Notwithstanding some evident numerical inaccuracies in approximating
these singular structures and computational difficulties encountered
in the solution of the highly ill-posed linear systems arising from
the discretization of the elliptic equation, long-time numerical
solutions seem to invariably reach a state of
equilibrium. Correspondingly, depending on the source term, the
spatial distributions of the approximated transport density
$\OptTdensH$ seem to converge to fractal-like or low-dimensional
structures. These singular configurations are shown to be sensitive to
initial conditions, probably corresponding to local minima of a
non-convex \LCF \added{, being the sum of a convex and a
  concave functional, $\Ene$ and $\Wmass_{\Pflux}$} . On the other hand, all the numerical
experiments presented in this paper show how the supports of
$\OptTdensH$ have the structure of an acyclic graph connecting the
supports of $\Source$ and $\Sink$.  The absence of loops is a
fundamental characteristic of the solution of the \BTP\, never imposed
a priori in our model, that seems to suggest yet another relationship
between the long-time numerical solution of DMK and the \BT\ transport
solution.

Only sparse examples in the literature addressing the numerical
solution of the \BT\ problem exist, both in the discrete
\cite{Xia:2010} and in the continuous settings
\cite{Oudet-Santambrogio:2011}.  In this last work the authors
introduce a $\varepsilon$-relaxed functional $\Gamma$-converging to
the minimizer of problem in~\cref{eq:min-divergence} as $\varepsilon
\to 0$.  A conjugate Gradient algorithm combined with finite
difference spatial discretizations is used to find sequence of
approximated minimizer. This approach has been showed to able to solve
the \BTP\ in simple problem settings, but still it suffers of the main
problems we are facing in our approach ( e.g., parameters tuning,
convergence to local minima).

In conclusion of this introduction, we would like to remark the
similarities between the model proposed in our paper and the revisited
version of the PP model presented in the discrete setting
in~\cite{Hu-Cai:2013} and, in a continuous setting, in the model
of~\cite{Haskovec-et-al:2015}, where the modulating exponent $\Pflux$
is moved from the flux to the decay term of~\cref{eq:sys-pflux}.  In a
recent analysis, \cite{Burger-et-al:2018} studies a model that is
similar to our proposed dynamics under special conditions and connects
it to discrete optimization problems addressing \CTP\ and \BTP.
\added{Moreover, our work seems to be related to the
  work in~\cite{Xia:2014} on the discrete $\Plapl$-Laplace with
  $\Plapl<0$.}

In this paper we want to highlight the optimization capabilities of
the model that, despite the above mentioned limitations, seems to be
well-suited for the solution of optimal (branched and congested)
transport problems without imposing any a-priory graph topology.
This is particularly relevant in the case of \BTP\ where the
topology of the optimal solution is actually the main unknown. Our
results show that the relaxation introduced by the time-dependency
in our extended \DMK\ allows a relatively easy numerical formulation
that is efficient and robust.
%

\section{\LCF\ and its minimization}
\label{sec:lyap-pflux}
The local existence result obtained in~\cite{Facca-et-al:2018} could
be extended to the case $\Pflux>1$, but, numerical simulations and
theoretical considerations suggest that the assumption of
H\"older-continuity of $\Tdens$ needs to be relaxed. In this case any
existence result seem out of reach at this time. However, being mostly
concerned with the asymptotic behavior of $(\Tdens(t),\Pot(t))$
solution of~\cref{eq:sys-pflux}, in this work we assume existence and
uniqueness of a solution pair for all $t\geq 0$, and present the
formal derivation of the \LCF\ for all $\Pflux>0$, as given by the
following Proposition.
\begin{Prop}
  \label{prop-decr-lyap-pflux}
  Assume that there exists $\bar{t}>0$ such that~\cref{eq:sys-pflux}
  admits a solution pair $\left(\Tdens(t),\Pot(t)\right)$ with
  $\Cont[1]$-regularity in time for $t\in[0,\bar{t}[$. Then the \LCF\
  $\Lyap_{\Pflux}$ given in~\cref{eq:lyap} is strictly decreasing
  along the $\Tdens(t)$ trajectories for all $\beta>0$, and its Lie
  derivative is given by:
  \begin{multline}\label{eq:lie-der}
    \frac{d}{dt}\left(\Lyap_{\Pflux}(\Tdens(t)) \right)
    =
    \\
    - \frac{1}{2}
    \int_{\Omega}{\Tdens(t)^{\Pflux}
      \left(\ABS{\Grad\PotOp(\Tdens(t))}^{\Pflux}
        - \Tdens^{\left(\frac{1-\Pflux}{\Pflux}\right)\Pflux}(t)
      \right)
      \left(
        \ABS{\Grad\PotOp(\Tdens(t))}^2
        - \left(\Tdens^{\frac{1-\Pflux}{\Pflux}}(t)\right)^2
      \right)
    }\dx .
  \end{multline}
\end{Prop}
\begin{proof}
  The proof is based on the equality
  \begin{equation*}
    \frac{d}{dt}\left(\Ene(\Tdens(t))\right)
    = -\frac{1}{2}
    \int_{\Domain}{\Dt{\Tdens}(t)\ABS{\Grad\PotOf{\Tdens(t)}}^2}\dx ,
  \end{equation*}
  proved in~\cite{Facca-et-al:2018} under the assumption of existence
  and uniqueness of the solution pair $(\Tdens(t),\Pot(t))$.  As a
  consequence, we can write:
  \begin{multline*}
      \frac{d}{dt}\left[\Lyap_{\Pflux}(\Tdens(t))\right]
      =\frac{d}{dt}\left[\Ene(\Tdens(t))+
        \Wmass_{\Pflux}(\Tdens(t))\right] = \\
      =
      -\frac{1}{2}
      \int_{\Domain}{
        \Dt{\Tdens}(t)
        \left(\ABS{\Grad \PotOf{\Tdens(t)}}^2
          -\Tdens^{\frac{2-\Pflux}{\Pflux}-1}(t)
        \right)
      }\dx .
  \end{multline*}
  Substituting $\Dt{\Tdens}(t)$ defined by~\cref{eq:sys-pflux-dyn} in
  the previous equation, using for simplicity $\Pot(t)$ in place of
  $\PotOf{\Tdens(t)}$ and rearranging terms, we can write:
  \begin{align*}
    \frac{d}{dt}\Lyap_{\Pflux}(\Tdens(t))
    &=
      -
      \frac{1}{2}
      \int_{\Domain}{
      \left(
        \left(\Tdens(t)\ABS{\Grad \Pot(t)}\right)^{\Pflux}
        -  
        \Tdens(t)
      \right)
      \left(\ABS{\Grad \Pot(t)}^2-\Tdens^{\frac{2-\Pflux}{\Pflux}-1}(t)\right)
      }\dx
    \\
    &
    \!  = - \frac{1}{2}
      \int_{\Domain}{
      \Tdens(t)^{\Pflux}
      \left(\ABS{\Grad  \Pot(t)}^{\Pflux}
         - \left(\Tdens^{\frac{1-\Pflux}{\Pflux}}(t)\right)^{\Pflux}
      \right)
      \left(
         \ABS{\Grad \Pot(t)}^2
         -
         \left(
           \Tdens^{\frac{1-\Pflux}{\Pflux}}(t)
         \right)^2
      \right)
      }\dx ,
  \end{align*}
  which shows that the derivative of $\Lyap_{\Pflux}$ along the
  $\Tdens(t)$-trajectories is strictly negative, since $\Tdens(t)$ is
  always strictly greater than zero when starting from a positive
  initial condition $\TdensIni>0$.\added{ The same computations holds
    also for the case $\Pflux=2$.}
\end{proof}
From~\cref{eq:lie-der} we can formally deduce that the derivative of
$\Lyap_{\Pflux}$ is equal to zero only when evaluated at the pair
$(\Equi{\Tdens},\Equi{\Pot})$ solution of:
\begin{equation}
  \label{eq:steady}
  \left\{
    \begin{aligned}
      &-\Div\left(\Equi{\Tdens}\Grad \Equi{\Pot}\right)
      =\Forcing , \\
      &\Equi{\Tdens}=\ABS{\Grad \Equi{\Pot}}^{\frac{\Pflux}{1-\Pflux}} 
      \quad  \text{on} \quad
      \left\{\Equi{\Tdens}>0\right\},
    \end{aligned}
  \right. 
\end{equation}
\added{when $\Pflux\neq 1$}.
It is worth pointing out that these last equations coincide with those
we would obtain by imposing $\Dt{\Tdens}(t)=0$
in~\cref{eq:sys-pflux-dyn}.  Moreover~\cref{eq:steady} immediately
suggests a link between the large-time equilibrium state
of~\cref{eq:sys-pflux} and the $\Plapl$-Poisson equation
\begin{equation*}
  -\Div\left(
    \ABS{\Grad \Potp}^{\Plapl-2}\Grad \Pot_{\Plapl}
  \right)
  = \Forcing ,
\end{equation*}
if the following relation between the exponents $\Pflux$ and $\Plapl$
holds
\begin{equation*}
  \Plapl-2=\frac{\Pflux}{1-\Pflux} .
\end{equation*}
However, since we are not able to provide rigorous proofs for all
values of $\Pflux>0$, 
we analyze separately the cases $0<\Pflux<1$ and $\Pflux>1$,
which, as it will be seen later, are related to the Congested and
Branched Transport Problems, respectively.

\subsection{Case $0<\Pflux<1$} 
In this instance we are able to show that the minimum of the \LCF\
$\Lyap_\Pflux$ is related to the solution of a $\Plapl$-Poisson
equation, as stated in the following Proposition.
\begin{Prop} \label{prop-min-lyap-pflux}
  Let $\Pwmass=\Pwmass[\Pflux]=(2-\Pflux)/\Pflux$,
  $\PPP=\PPP(\Pflux)=2-\Pflux$, $0<\Pflux<1$, $\Lyap_{\Pflux}$ be
  defined as in~\cref{eq:lyap}, and $\Lplus{\Pwmass}(\Domain)$ be the space
  of non-negative functions in $\Lspace{\Pwmass}(\Domain)$.  Given
  $\Source,\Sink \in \Lspace{2}(\Domain)$ with equal mass, then
  \begin{equation}  \label{eq:min-lyap-pflux}
    \inf_{\Tdens \in \Lplus{\Pwmass}(\Domain)}\Lyap_{\Pflux}(\Tdens)
    =\inf_{\Vel\in \Vof{\Lspace{\PPP}(\Domain)}{\Dim}}
    \left\{
      \int_{\Omega}{\frac{\ABS{\Vel}^{\PPP}}{\PPP}}\dx
      \ : \ \Div\Vel = \Forcing\deleted{=\Source-\Sink}
    \right\} .
  \end{equation}
  Moreover, the functional $\Lyap_{\Pflux}$ admits a unique minimizer
  $\Equi{\Tdens_{\Pflux}}\in\Lplus{\Pwmass}(\Domain)$ given by
  $\Equi{\Tdens_{\Pflux}}=\ABS{\Grad \Potp}^{\Plapl-2}$ where $\Pot_p$
  is the weak solution of the $\Plapl$-Poisson equation
  \begin{equation}
    \label{eq:plapl-eqs}
    -\Div(\ABS{\Grad \Potp}^{\Plapl-2}\Grad\Potp)=
    \Forcing\deleted{=\Source-\Sink},
  \end{equation}
  with $\Plapl$ the conjugate exponent of $\PPP$, i.e.,
  $\Plapl=\frac{2-\Pflux}{1-\Pflux}$.
\end{Prop}
\begin{proof} 
  \label{proof-min-lyap2}
  We start the proof recalling the following variational
  characterization of the energy functional $\Ene$ for a general
  non-negative measure $\Tdens$~\citep{Bouchitte-et-al:1997}:
  \begin{align}
    \label{eq:energy-duality}
    \Ene(\Tdens)
    &=
    \sup_{\varphi \in \PLaplSpace}
    \!\left\{
      \int_{\Domain}\varphi\Forcing\dx-
      \int_{\Domain}\frac{\ABS{\Grad\varphi}^2}{2}\Tdens\dx
    \right\}
    \\
    &=
    \inf_{\Field \in \Vof{\Lspace[\Tdens]{2}(\Domain)}{\Dim}}
    \!\left\{
      \int_{\Domain}\frac{\ABS{\Field}^2}{2}\Tdens\dx
      \,:\,
      \Div(\Tdens \Field ) = \Forcing
    \right\},\nonumber
  \end{align}
  where $\PLaplSpace$ is the classical Sobolev space and, with some
  abuse of notation,
  $\Lspace[\Tdens]{2}=\{v: \int_{\Omega} v^2\Tdens\dx<\infty\}$.  Note
  that in the above result the divergence constraint is considered in
  the sense of distributions.  Then, we can rewrite
  $\Lyap_{\Pflux}(\Tdens)$ as:
  \begin{gather}
    \label{eq:lyap-theta-pflux}
    \Lyap_{\Pflux}(\Tdens)
    =
    \inf_{\Field\in\Vof{\Lspace[\Tdens]{2}(\Domain)}{\Dim}}
    \left\{
      \Theta_{\Pflux}(\Tdens,\Field)
      \  : 
        \Div(\Tdens \Field) = \Forcing
    \right\}
    \quad
    \forall\Tdens\in\Lplus{\Pwmass[\Pflux]}(\Domain) ,
  \end{gather}
  where
  \begin{equation*}
    \Theta_{\Pflux}(\Tdens,\Field):=
    \frac{1}{2}
    \int_{\Domain}{\ABS{\Field}^{2}\Tdens}\dx
    +
    \frac{1}{2} 
    \int_{\Domain}{
      \frac{
        \Tdens^{\frac{2-\Pflux}{\Pflux}}
      }{
        \frac{2-\Pflux}{\Pflux}
      }
    }
    \dx .
  \end{equation*}
  For any pair
  $(\Tdens,\Field)\in\left(\Lplus{\Pwmass}(\Domain),
    \Vof{\Lspace[\Tdens]{2}(\Domain)}{\Dim}\right)$ and
  for $\PPP \in ]1,2[$, we use Young inequality with
  conjugate exponents $2/\PPP$ to obtain: 
  \begin{equation*}
    \int_{\Domain}{
      \ABS{\Field\Tdens}^{\PPP}
    }\dx
    =
    \int_{\Domain}{
      \ABS{\Field}^{\PPP}\Tdens^{\PPP/2}\Tdens^{\PPP/2}
    }\dx
    \leq
    \frac{\PPP}{2}
    \int_{\Domain}{
      \ABS{\Field}^2\Tdens
    }\dx
    +
    \frac{2-\PPP}{2}
    \int_{\Domain}{
      (\Tdens^{\frac{\PPP}{2}})^{\frac{2}{2-\PPP}}
    }\dx .
  \end{equation*}
  Since $\frac{\PPP}{2-\PPP}=\frac{2-\Pflux}{\Pflux}$, which yields the
  relation $\PPP=\PPP(\Pflux)=2-\Pflux$, dividing by $\PPP$ we
  can rewrite the previous inequality as:
  \begin{equation}\label{eq:young}
    \int_{\Domain}{
      \frac{\ABS{\Field\Tdens}^{(2-\Pflux)}}{(2-\Pflux)}
    }\dx
    \leq
    \frac{1}{2}
    \int_{\Domain}{
      \ABS{\Field}^2\Tdens
    }\dx
    +
    \frac{1}{2}
    \int_{\Domain}{
      \frac{\Tdens^{\frac{2-\Pflux}{\Pflux}}}{\frac{2-\Pflux}{\Pflux}}
    }\dx
    = \Theta_{\Pflux}(\Tdens,\Field) ,
  \end{equation}
  which holds for all $\Tdens\in\Lplus{\Pwmass[\Pflux]}(\Domain)$ and
  all $\Field\in\Vof{\Lspace[\Tdens]{2}(\Domain)}{\Dim}$.  Now we take
  the infimum on both sides of the previous equation over the
  $\Tdens$-divergence constrained
  $\Field \in \Vof{\Lspace[\Tdens]{2}(\Domain)}{\Dim}$ and use
  characterization~\cref{eq:lyap-theta-pflux} to obtain:
  \begin{align*}
    \inf_{\Field \in\Vof{\Lspace[\Tdens]{2}(\Domain)}{\Dim}}
    \left\{
      \int_{\Domain}{
        \frac{\ABS{\Field\Tdens}^{(2-\Pflux)}}{(2-\Pflux)}
      }\dx
      \ : 
        \Div(\Tdens \Field) =\Forcing
    \right\}
    \leq
    \Lyap_{\Pflux}(\Tdens)
    \quad
    \forall \Tdens \in \Lplus{\Pwmass[\Pflux]}(\Domain) .
  \end{align*}
  Taking the infimum over all
  $\Tdens \in \Lplus{\Pwmass[\Pflux]}(\Domain)$ yields:
  \begin{equation}
    \label{eq:infinf}
    \inf_{\Tdens \in \Lplus{\Pwmass[\Pflux]}(\Domain)}
    \left\{
      \inf_{\Field \in\Vof{\Lspace[\Tdens]{2}(\Domain)}{\Dim}}
      \left[
        \int_{\Domain}{
          \frac{\ABS{\Field\Tdens}^{(2-\Pflux)}}{(2-\Pflux)}
        }\dx
        \ 
        : 
        \Div(\Tdens \Field) =\Forcing
      \right]
    \right\}
    \leq
    \inf_{\Tdens \in \Lplus{\Pwmass[\Pflux]}(\Domain)}
    \Lyap_{\Pflux}(\Tdens) .
  \end{equation}
  According to the relationship between the solution of the
  $\Plapl$-Poisson equation and \CTP\, we have that for any $\PPP>1$
  the optimal vector field is given by
  \citep[ex.~2.2,~Chapter~4]{Ekeland-Teman:1999}:
  \begin{equation*}
    \Argmin_{ \Vel\in \Vof{\Lspace{\PPP}(\Domain)}{\Dim}}
    \left\{
      \int_{\Domain}{
        \frac{\ABS{\Vel}^\PPP}{\PPP}
      }\dx
      \ : 
      \Div(\Vel)=\Forcing
    \right\}
    =
    \Equi{\Vel}
    =
    -\ABS{\Grad\Pot_{\Plapl}}^{\Plapl-2}\Grad \Pot_{\Plapl} ,
  \end{equation*}
  where $\Plapl$ is the conjugate exponent of $\PPP$.  Now,
  for $\PPP=2-\Pflux$ (and thus $\Plapl=(2-\Pflux)/(1-\Pflux)$)
  and using~\cref{eq:infinf}), we obtain:
  \begin{align}
    \label{eq:ineq}
    \int_{\Domain}{
      \frac{\ABS{\Grad\Pot_{\Plapl}}^{\Plapl}}{2-\Pflux}
    }\dx
    &=
    \int_{\Domain}{
      \frac{\ABS{\Equi{\Vel}}^{(2-\Pflux)}}{2-\Pflux}
    }\dx
    =\!\!\!
    \inf_{\Vel \in \Vof{\Lspace{(2-\Pflux)}(\Domain)}{\Dim}}
    \left\{
      \int_{\Domain}{
        \frac{\ABS{\Vel}^{(2-\Pflux)}}{2-\Pflux}
      }\dx
      \ 
      : 
      \Div(\Vel) =\Forcing
    \right\}
    \\\nonumber
    &\leq
    \inf_{\Tdens \in \Lplus{\Pwmass[\Pflux]}(\Domain)}
    \left\{
      \inf_{\Field \in\Vof{\Lspace[\Tdens]{2}(\Domain)}{\Dim}}
      \left\{
        \int_{\Domain}{
          \frac{\ABS{\Field\Tdens}^{(2-\Pflux)}}{2-\Pflux}
        }\dx
        \ 
        : 
        \Div(\Tdens \Field) =\Forcing
      \right\}
    \right\}
    \\\nonumber
    &\leq
    \inf_{\Tdens \in \Lplus{\Pwmass[\Pflux]}(\Domain)}
      \Lyap_{\Pflux}(\Tdens)
    \leq 
    \Lyap(
    \OptTdens_{\Pflux}
    )
    =\Ene(\OptTdens_{\Pflux})+\Wmass_{\Pflux}(\OptTdens_{\Pflux}). 
  \end{align}
  We can compute the term $\Ene(\OptTdens_{\Pflux})$ as follows.
  By~\cref{eq:energy-duality}, the following chain of inequalities is
  obtained:
  \begin{gather*}
    \int_{\Domain}{
    \left(\Forcing\Potp-\OptTdens_{\Pflux}\frac{\ABS{\Grad\Potp}^2}{2}\right)
  }\dx
    \leq
    \sup_{\varphi\in\Sob[1]{\Plapl}(\Domain)}
    \int_{\Domain}{
      \left(
      \Forcing\Ftest-\OptTdens_{\Pflux}\frac{\ABS{\Grad\Ftest}^2}{2}
      \right)
    }\dx
    \\
    =
    \Ene(\OptTdens_{\Pflux})
    =
    \inf_{\Field \in \Vof{\Lspace[\OptTdens_{\Pflux}]{2}(\Domain)}{\Dim}}
    \left\{
      \int_{\Domain}{
        \frac{\ABS{\Field}^2}{2}\OptTdens_{\Pflux}
      }\dx
      \ :  
      \Div(\OptTdens_{\Pflux} \Field) = \Forcing
    \right\}
    \leq 
    \int_{\Domain}{
      \OptTdens_{\Pflux}\frac{\ABS{\Grad\Potp}^2}{2}
    }\dx ,
  \end{gather*}
  Noting that
  $\int_{\Domain}{\Forcing\Potp}\dx=\int_{\Domain}{\ABS{\Grad\Potp}^\Plapl}\dx$ 
  holds since $\Potp$ is the weak solution of~\cref{eq:plapl-eqs}, 
  using $\OptTdens_{\Pflux}=\ABS{\Grad\Potp}^{\Plapl-2}$, we obtain
  $\Ene(\OptTdens_{\Pflux})=\int_{\Domain}{\frac{\ABS{\Grad\Potp}^\Plapl}{2}}\dx$. 
  Using this equality in~\cref{eq:ineq} and noting that
  $\Plapl=(\Plapl-2)(2-\Pflux)/\Pflux$, we obtain: 
  \begin{equation*}
    \Lyap_{\Pflux}(\OptTdens_\Pflux)=\Ene(\OptTdens_\Pflux)
    +
    \Wmass_{\Pflux}(\OptTdens_\Pflux)
    =\int_{\Domain}{
      \frac{\ABS{\Grad\Pot_{\Plapl}}^{\Plapl}}{2-\Pflux}
    }\dx .
  \end{equation*}
  Thus, all the inequalities in~\cref{eq:ineq} are actually
  equalities, and we can conclude that $\Equi{\Tdens}_{\Pflux}$ is a
  minimum, which is unique since the functional $\Lyap_{\Pflux}$ is
  strictly convex.  This last assertion follows from the observation
  that $\Ene$ is convex, being the supremum of functionals that are
  linear with respect to $\Tdens$, and $\Wmass_{\Pflux}$ is strictly
  convex for $0<\Pflux<1$.
\end{proof}
\begin{Remark}
  The assumption $\Source,\Sink \in \Lspace{2}(\Domain)$ can be 
  relaxed whenever $\int_{\Domain}\Forcing \Ftest\dx $ is
  well defined for all $\Ftest\in \Sob[1]{\Plapl}(\Domain)$.
\end{Remark}
\Cref{prop-decr-lyap-pflux,prop-min-lyap-pflux} suggest the
formulation of the following conjecture for the case $0<\Pflux<1$:
\begin{Conject}
  \label{conj:pflux-lt-1}
  For $0\le\Pflux\le1$ and for any initial data $\TdensIni$, the pair
  $(\Tdens(t),\Pot(t))$, solution of the extended dynamic
  Monge-Kantorovich equations~\cref{eq:sys-pflux}, converges to the
  pair $(\ABS{\Grad\Potp}^{\Plapl-2},\Potp)$, where $\Pot_p$ is the
  solution of the $\Plapl$-Poisson equation with
  \begin{equation} \label{eq:pflux-plapl}
    \Plapl=\frac{2-\Pflux}{1-\Pflux}.
  \end{equation}
\end{Conject}

\begin{Remark}
  We want to emphasize two remarkable facts
  regarding~\cref{conj:pflux-lt-1}.  For $\Pflux\Tendsto 1$ the
  exponent of the $\Plapl$-Laplacian tends to infinity
  ($\Plapl\Tendsto +\infty$), coherently with the fact that the
  \MKEQS\ are the limit of the $\Plapl$-Poisson problem, as already
  shown, e.g., in~\cite{Evans-Gangbo:1999}.  At the same time, the
  exponent $\PPP=2-\Pflux$ tends to $1$, in agreement with the
  equivalence between the \MKEQS\ and Beckmann Problem.  When
  $\Pflux\Tendsto 0$ then $\Tdens(t)\Tendsto 1$ and $\Pot(t)$
  converges to the solution of a classical Poisson problem
  ($\Plapl=2$).  Thus it is possible to include also the value
  $\Pflux=0$ in~\cref{conj:pflux-lt-1},
\end{Remark}

\subsection{Case $\Pflux>1$}
\label{sec:pflux-gt-1}

In this section we discuss our attempts to extend the arguments
presented in the previous section to the \BT\ problem.  We are
particularly interested in understanding if the functional
$\Lyap_\Pflux(\Tdens(t))$, which by \cref{prop-decr-lyap-pflux} is
decreasing in time, admits a minimum as $t\Tendsto\infty$ and that
this minimum is actually attained at $\OptTdens$.  However, in this
case the \LCF\ is strongly non-convex, suggesting that several local
minima exist, as indicated also by the numerical results reported in
the next section that show strong dependence upon the initial data
$\TdensIni$.  Mostly formal calculations and several numerical
experiments consistently point towards the existence of a strong
connection between \BTP\ and the proposed formulation, but we are not
able to exactly identify the \BTP\ equivalent to the minimization of
$\Lyap_\Pflux$.  In fact, the first part of the proof
of~\cref{prop-min-lyap-pflux} remains valid, at least for the case
$1<\Pflux<2$, in which the exponent $\PPP$ remains positive.  Indeed,
from~\cref{eq:ineq} the following sequence of inequalities is derived:
\begin{align*}
  \inf_{\Vel \in \Vof{\Lspace{\PPP}(\Domain)}{\Dim}}&
  \left\{
    \int_{\Domain}{\frac{\ABS{\Vel}^{\PPP}}{\PPP}}\dx
     \ :  \Div(\Vel) =\Forcing
  \right\}
  \\
  \nonumber
  \leq
  &\inf_{\Tdens,\Field}
  \left\{
    \int_{\Domain}{\frac{\ABS{\Field\Tdens}^{\PPP}}{\PPP}}\dx
    \ : (\Tdens,\Field)\in
    \Lplus{\frac{2-\Pflux}{\Pflux}}(\Domain)\times
        \Vof{\Lspace[\Tdens]{2}(\Domain)}{\Dim},
    \quad
    \Div(\Field \Tdens) =\Forcing
  \right\}
  \\
  \nonumber
  \leq
  &\inf_{\Tdens  \in \Lplus{\Pwmass[\Pflux]}(\Domain)}\Lyap_{\Pflux}(\Tdens).
\end{align*}
\replaced{This optimization problem resembles the \BTP\ formulation in
  \cref{eq:branch-intro}.}{
  This optimization problem resembles the \BTP\ proposed
  by~\cite{Xia:2003}, with the exponent $\PPP$ playing the role of the
branch exponent $\Pbranch$.}  However, we have to highlight an
important difference between the two problems. The formulation
in~\cite{Xia:2003} uses integrals computed with respect to a Hausdorff
measure that well adapts to the singular structures arising in
\BT. Our computations, on the other hand, are made always with
respect to the Lebesgue measure, indicating that relation
$\PPP=2-\Pflux$ does not hold.  
Moreover, for $\Pflux\geq 2$, $\Lyap_{\Pflux}$ is not well defined for
$\Tdens$ attaining zero on some regions of $\Domain$, while we expect
that the asymptotic $\OptTdens$ will be zero on large portions of the
domain.  These two elements suggest that a proper re-formulation of
the \LCF\ is required for $\Pflux>1$.  One possible strategy to
reconcile our inability to address singular measures is inspired by
the Modica-Mortola approach, effectively used
in~\cite{Santambrogio:2010,Oudet-Santambrogio:2011}.  The main idea is
to introduce a parameter $\Eps>0$ in~\cref{eq:young} and use Young
inequality with $\Eps$ as parameter in order to weight differently the
energy and mass terms, $\Ene(\Tdens)$ and $\Wmass_{\Pflux}(\Tdens)$,
in $\Lyap_{\Pflux}(\Tdens)$.  Intuitively, $\Eps$ should be raised to
a suitable power that allows the different energy and mass components
to scale correctly. The main difficulty lies in the identification of
the proper scaling power, but this identification is at the moment
still elusive.

Despite these difficulties, together with some not completely
understood theoretical issues, we present in the next section
numerical simulations that suggest that system~\cref{eq:sys-pflux}
admits a steady state equilibrium $(\Equi{\Tdens},\Equi{\Pot})$, where
the supports of the numerical solutions $\OptTdensH$ seem to
approximate the typically singular (low-dimensional) formations
emerging in \BT\ problems and the vector field
$\Equi{\Vel}=-\Equi{\Tdens}\Grad\Equi{\Pot}$ solve the \BTP\ as
formulated in~\cite{Xia:2003}. 

\section{Numerical solution of the Extended \DMK\ equations}
\label{sec:numeric-pflux}

In this section we report several two-dimensional numerical tests in
support of our conjectures. In particular we try to approximate
explicitly the steady state solution of~\cref{eq:sys-pflux} and look
at the qualitative behavior of the equilibrium configurations for
$t \Tendsto +\infty$ in the \CTP\ ($0<\Pflux<1$) and \BTP\
($\Pflux>1$) cases.
For $0<\Pflux<1$ we test \cref{conj:pflux-lt-1} comparing our results
with an exact solution $\Potp$ of the $\Plapl$-Poisson equation.  For
the case $\Pflux>1$, several experiments point to the existence of a
connection between the large-time solution of our model and \BTP\
solutions.  For different values of $\Pflux$ and varying types of
source terms $\Source$ and $\Sink$, including two largely different
test cases employing distributed and a point sources, our dynamics
invariably converges to an equilibrium point in correspondence of
which $\Tdens$ always displays branching singular structures.

\subsection{The Numerical Approach}\label{sec:num-met}

The discretization method is based on the Finite Element approach
described in~\cite{Facca-et-al:2018,Facca-et-al:sisc:2018}.
\Cref{eq:sys-pflux-dyn} is projected into a piecewise constant finite
dimensional space defined on a triangulation
$\Triang[\MeshPar](\Omega)$ of the domain $\Omega\subset\REAL^{\Dim}$.
The elliptic equation in~\cref{eq:sys-pflux-div} is discretized using
linear conforming Galerkin finite elements defined on a mesh
$\Triang[\MeshPar/2](\Omega)$ obtained by uniform refinement of
$\Triang[\MeshPar](\Omega)$ (approach called
$\PC{1,\MeshPar/2}-\PC{0,\MeshPar}$).  We choose this spatial
discretization method for its robustness and stability as shown
in~\cite{Facca-et-al:sisc:2018}.  Accordingly, the approximate pair
$(\TdensH,\PotH)$ can be written as:
\begin{gather*}
  \PotH(t,x)=\sum_{i=1}^{\Vdim} \Pot_{i}(t)\Vbase[i](x) \quad
  \Vbase[i]\in \PONE(\Triang[\MeshPar/2]) ,
  \\
  \TdensH (t,x)= \sum_{k=1}^{\Wdim}\Tdens_{k}(t) \Wbase[k](x) \quad
  \Wbase[k]\in \PZERO(\Triang[\MeshPar]) ,
\end{gather*}
where $\PONE(\Triang[\MeshPar/2])$ and $\PZERO(\Triang[\MeshPar])$ are
piecewise linear conforming and piecewise constant FEM spaces,
respectively.  The time discretization of the projected system is
obtained by means of forward Euler time stepping.  Denoting with
$\TdVec[\tstep]$ and $\UVec[\tstep]$ the vectors collecting the values
of $\TdensH$ on the triangles and of $\PotH$ on the nodes at the
$\tstep$-th time step, the following sequence of linear systems needs
to be solved:
\begin{subequations}\label{eq:fem}
  \begin{align}
    \label{eq:lin-sys-pflux} 
    &\Matr{A}[\TdVec[\tstep]] \UVec[\tstep] = \Vect{b} ,
    \\
    \label{eq:ee-pflux}
    &\TdVec[\tstep+1] = \TdVec[\tstep]+\Deltat
      \left[
        \Matr{B}_{\Pflux}[\UVec[\tstep]]
        \left(\TdVec[\tstep]\right)^{\Pflux}
        -\TdVec[\tstep]
      \right] ,
  \end{align}
\end{subequations}
where $\Matr{A}[\TdVec[\tstep]]$ is the stiffness matrix associated to
$\TdVec[\tstep]$ and $\Matr{B}_{\Pflux}[\UVec[\tstep]]$ is the matrix
defining the norm of the gradient of $\PotH(t^\tstep,x)$ raised to the
power $\Pflux$.  The time stepping is initiated from $\TdVec[0]$
obtained by projecting the initial data $\TdensIni$ on the
$\PZERO(\Triang[\MeshPar])$ space.  The above system is then iterated
until the norm of the relative variation between two consecutive
$\TdensH$ solutions is smaller than the tolerance $\TolTime$, i.e.:
\begin{equation*}
  \Var(\TdensH^{\tstep}):=
   \frac{\NORM{\TdensH^{\tstepp}-\TdensH^{\tstep}}_{L^2(\Domain)}} 
         {\Deltat\NORM{\TdensH^{\tstep}}_{L^2(\Domain)}})\le\TolTime .
\end{equation*}
When this occurs we assume that the equilibrium configuration has been
reached.

At each time step, the linear system in~\cref{eq:lin-sys-pflux} is
solved via Preconditioned Conjugate Gradient (\PCG) with an ad-hoc
preconditioning strategy.  This preconditioner, developed and
discussed in details in~\cite{Martinez-et-al:2017}, exploits the
time-stepping sequence to devise an efficient and robust deflation
strategy, whereby partial eigenpairs are evaluated to improve the
spectral properties of the preconditioned linear system. This approach
is fundamental to achieve convergence of the iterations in particular
for the case $\Pflux>1$.

\begin{Remark}
  It is natural to impose a lower bound on $\TdensH$ (say $10^{-10}$)
  to bound from below the smallest eigenvalue of the stiffness matrix
  in the FEM formulation and limit its condition number. This was
  indeed our first attempt. However, we quickly realized that our
  approach is not influenced by this lower bound. In fact, the
  solution of the linear system at each time step is the same and
  requires exactly the same number of PCG iterations irrespectively of
  the presence of the lower bound. Also the dynamics of $\TdensH$ and
  of the corresponding $\Lyap_{\Pflux}$ for $0<\Pflux<2$ are not
  affected by the lower bound. On the other hand, in the case
  $\Pflux\ge2$, the lower bound influences the value of the \LCF\,
  which tends to $-\infty$ as $\Tdens\rightarrow 0$.
\end{Remark}

\subsection{Numerical Experiments}
\subsubsection{Case $0<\Pflux\leq 1$}
\label{sec:numeric-pflux-lt-1}

\begin{figure}
  \centerline
  {
    \includegraphics[
      trim={170 0 0 0},clip,
      width=0.43\textwidth]{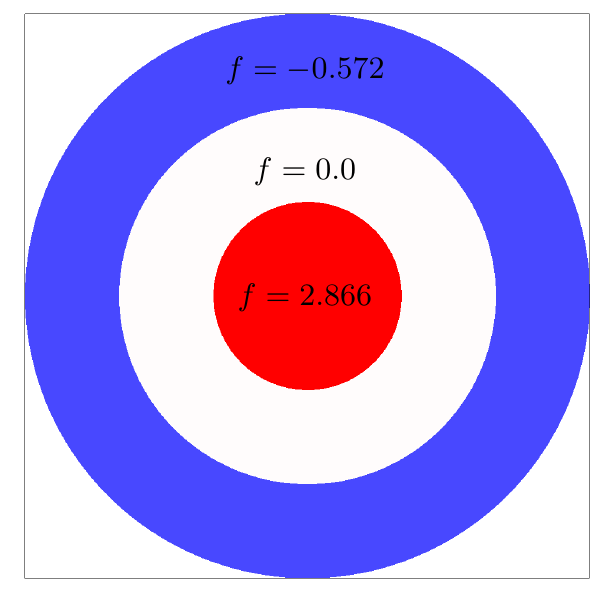}
     \includegraphics[width=0.57\textwidth]{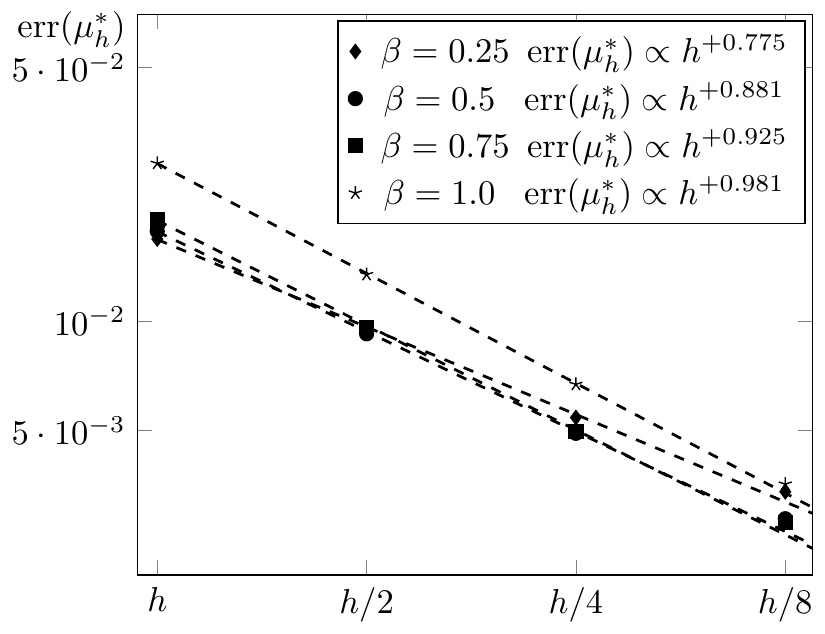}  
  }
  \caption{Left panel: \replaced{spatial distribution of
      the}{triangulation $\Triang_{\MeshPar}(\Omega)$, with 5191 nodes
      and 10179 triangles, and} forcing term
    $\Forcing(x,y)=\Fradial(r)$.  The mesh points lie on the
    concentric circles that form the boundary of the supports of
    $\Source$ and $\Sink$.  Center panel: the spatial distribution of
    $\Source$ and $\Sink$. Right panel: FEM experimental convergence
    for $\Pflux=0.25, 0.5, 0.75, 1.0$.  The legend above the figure
    reports the experimental convergence rates for each value of
    $\Pflux$. }
  \label{fig:mesh-forcing-plapl}
\end{figure}

In this series of tests we compare the long-time limit of $\TdensH$,
denoted with $\OptTdensH$, against
$\Equi{\Tdens_{\Pflux}}:=\ABS{\Grad\Pot_{\Plapl}}^{\Plapl-2}$, where
$\Pot_{\Plapl}$ is the solution of the $\Plapl$-Poisson equation, for
which an explicit formula is known, and
$\Plapl=(2-\Pflux)/(1-\Pflux)$, as given by~\cref{eq:pflux-plapl}
of~\cref{conj:pflux-lt-1}.  We consider a two dimensional example
taken from~\cite{Barrett-Liu:1993}, where the radially symmetric
forcing term is given by $\Forcing(x,y)=\Fradial(r)$ with
$r=\sqrt{x^2+y^2}$ and $\Fradial:]0,1[\mapsto \REAL$.  Under these
assumptions, the exact solution of the $\Plapl$-Poisson equation is:
\begin{equation*}
  \Pot_{\Plapl}(x,y)=U(r)=-\int_{r}^{1}
  \Sign (Z(t))\ABS{Z(t)}^{\frac{1}{\Plapl-1}}dt
  \quad
  Z(r)=-\frac{1}{r}\int_{0}^{r}t\Fradial(t)dt .
\end{equation*}
According to the relation between $\Plapl$ and $\Pflux$, we can write
the following explicit formula for $\Equi{\Tdens}_{\Pflux}$:
\begin{equation} \label{eq:exact-plapl-tdens}
  \Equi{\Tdens}_{\Pflux}(x,y)=
  \ABS{Z(r)}^{\frac{\Plapl-2}{\Plapl-1}}=\ABS{Z(r)}^{\Pflux}
\end{equation}
Note that this $\Equi{\Tdens}_{\Pflux}$ is
well defined also for $\Pflux=1$, a value corresponding to the case
$\Plapl=+\infty$.  This optimal density corresponds to the
$\OT$-density solution of the \MKEQS, a problem already considered
within our setting
in~\citet{Facca-et-al:2018,Facca-et-al:sisc:2018}.

In our numerical experiments we take $\Fradial$ as a piecewise
constant function, positive in the interval $]0,1/3[$, zero in
$[1/3,2/3]$, and negative in $]2/3,1[$.  The value of $\Fradial$ on
the positive and the negative parts is given by two constants $c_1$
and $c_2$ that are calculated to maintain orthogonality (up to machine
precision) of the right hand side $\Vect{b}$ of the linear
system~\cref{eq:lin-sys-pflux} with respect to the constant vectors.
This tuning corrects for quadrature errors and is necessary to provide
accurate approximations of the integrals and a-priori exclude errors
and inaccuracies introduced by the piecewise representations of the
circles.  The mesh $\Triang[\MeshPar]$ and the forcing term $\Forcing$
are shown in~\cref{fig:mesh-forcing-plapl}.

The numerical experiments consist in testing the existence of a steady
state $\OptTdensH$ for different values of $\Pflux$ (0.25, 0.5, 0.75,
1.0) and evaluating the error with respect to the candidate exact
solution 
%
  $\Err(\Tdens):=
  {\NORM{\Tdens-\OptTdens_{\Pflux}}_{\Lspace{2}(\Domain)}}/
       {\NORM{\OptTdens_{\Pflux}}_{\Lspace{2}(\Domain)}} .$
%
We solve the extended \DMK\ equations on a sequence of uniformly
refined grids and evaluate the experimental convergence rate by means
of $\Err(\OptTdensH)$. The sequence is built by uniform refinement of
the initial unstructured mesh $\Triang[\MeshPar](\Omega)$
\deleted{shown in
}, characterized
by 5191 nodes and 10179 triangles \added{(we generate it using
  Python package Meshpy)}.  Convergence in time is tested by looking
at the evolution of $\Var(\TdensH(t))$ and $\Err(\TdensH(t))$.
Steady-state is considered achieved when $\TolTime=5\times 10^{-7}$.

\begin{figure}
  \centerline
  {
     \includegraphics[width=0.9\textwidth]
     {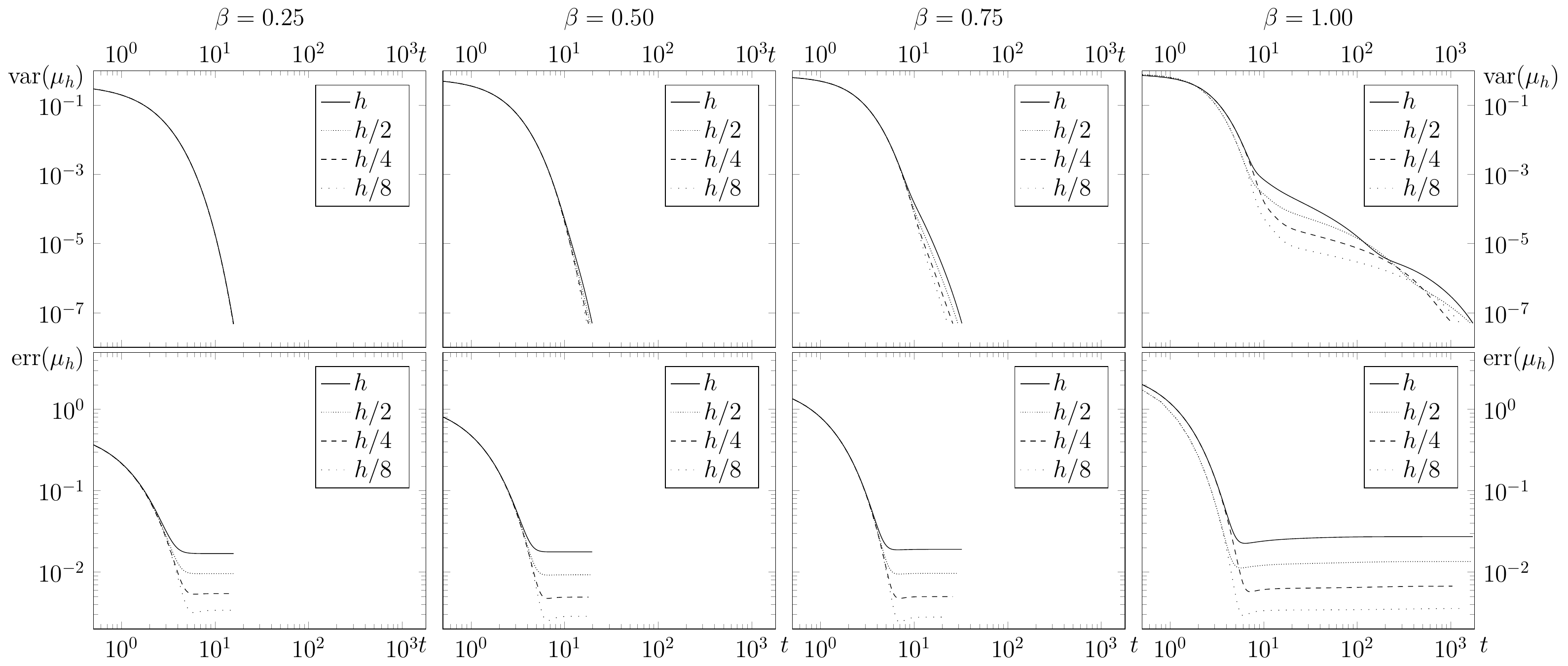}
  }
  \caption{
    Log-log plots of $\Var(\TdensH(t,\cdot))$ (upper panels)
    and $\Err(\TdensH(t,\cdot))$ (lower panels) vs. time.
    The columns refer, from left to right, to the
    results obtained with $\Pflux=0.25, 0.5, 0.75, 1.0$.
  }
  \label{fig:var-err-plapl}
\end{figure}

\begin{figure}
  \centerline
  {
     \includegraphics[width=0.9\textwidth]{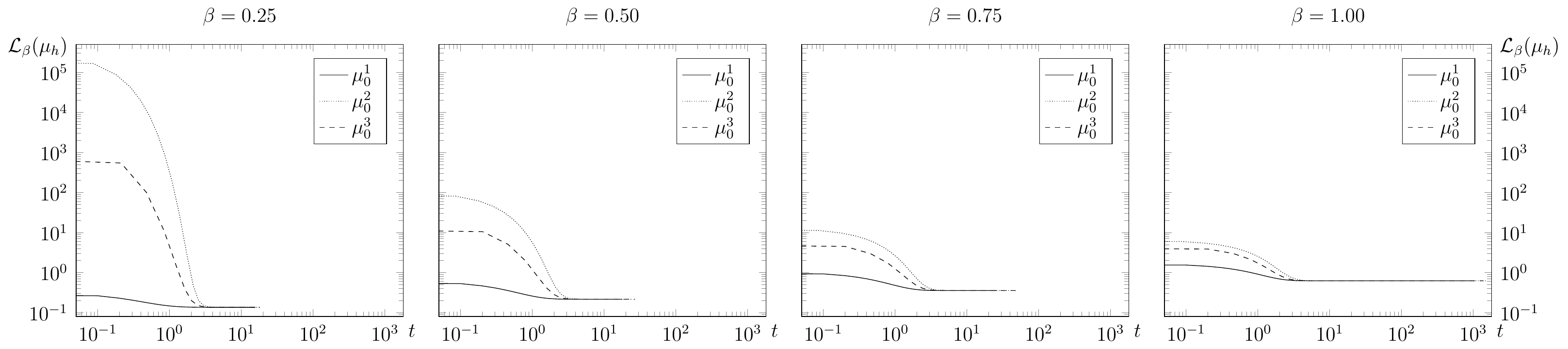}  
  }
  \caption
  { Time behavior of the \LCF\ $\Lyap_{\Pflux}(\TdensH(t)$, for
    $\Pflux=0.25, 0.5, 0.75, 1.0$ (from left to right) starting from
    three different initial data $\TdensIni$.  We report the results
    only for the coarser mesh as they do not seem to depend on mesh
    resolution.} 
  \label{fig:lyap-plapl}
\end{figure}

The experimental convergence rates are shown
in~\Cref{fig:mesh-forcing-plapl}, right panel, and vary in the range
$m=0.775\div0.981$ for $\Pflux=0.25\div1$, respectively, thus
displaying optimal convergence of the spatial discretization.
\Cref{fig:var-err-plapl} shows the log-log plot of the time-variation
$\Var(\TdensH(t))$ (top row) and the relative error $\Err(\TdensH(t))$
(bottom row) as a function of time for the same values of $\Pflux$.
From the first set of plots we can see that, as time increases, the
variation tends towards zero as power-law with a rate that is
independent of the mesh level and decreases as the power $\Pflux$
increases. In other words, for any tested mesh, the smaller $\Pflux$
the faster the equilibrium configuration is reached.  The case
$\Pflux=1$ shows the slowest convergence towards steady-state and some
influence of the mesh resolution appears. This is an evident signal of
the difficulty of the MK problem.  The relative error
(\Cref{fig:var-err-plapl}, lower row) stagnates at a relatively small
time reaching values that decrease at a constant factor with the mesh
level, coherently with the experimental convergence rates previously
calculated. Note that, for all practical purposes, the time tolerance
$\TolTime$ could be increased to much bigger values without affecting
the $\Err(\OptTdensH)$, quantity that remains essentially stationary
in all simulations after $t=10$, i.e., when
$\Var(\TdensH(t)\in [10^{-3},10^{-4}]$.  However, for reasons of
numerical testing, all our simulations are continued until the
indicated tolerance $\TolTime$ is achieved.

To conclude our exploration of this case, we look at the time
evolution of the \LCF\ $\Lyap_{\Pflux}(\TdensH(t))$ starting from
three different initial data $\TdensHIni[i]\ (i=1,2,3)$.
\Cref{fig:lyap-plapl} shows the numerical results obtained for the
uniform initial condition $\TdensHIni[1]=1$ and for the
$\TdensHIni[2,3]$ distributions reported later in the left panel
of~\cref{fig:final-pflux-150-td23}.  In all simulations
$\Lyap_{\Pflux}(\TdensH(t))$ decreases monotonically and always
attains the same minimum value independently of the initial
conditions. For all the starting points, the value of
$\Lyap_{\Pflux}(\TdensH(t))$ becomes numerically stationary before
$t=10$. However, its value continues to decrease but at progressively
lower rates.  Overall, these results provide convincing support of the
correctness of~\cref{conj:pflux-lt-1}.

\subsubsection{Case $\Pflux>1$}
\label{sec:numeric-pflux-gt-1}

In this section we discuss our numerical results related to ramified
transport by looking at some qualitative features that the solutions
emerging from our proposed model share with more classical \BT\
solutions reported e.g. in~\citet{Xia:2015}. We explore the numerical
features of the discretization algorithm and its robustness by varying
the exponent $\Pflux$, the mesh size parameter $\MeshPar$, and the
initial conditions $\TdensHIni$.  We look at the time-convergence of
the solution towards an equilibrium point and at the behavior of the
\LCF\ $\Lyap_\Pflux$ as $t\Tendsto\infty$.  All these results are
critically assessed and are here submitted in support of our
conjecture on the connection between the asymptotic
configuration of our dynamics and the solution of the \BTP.

To this aim, we consider three different test cases, TC1, TC2, and TC3
defined on the same domain $\Omega=[0,1]\times[0,1]$ and characterized
by varying the forcing function.  In TC1, $\Source$ and $\Sink$ are
two constant functions with equal value and supports in the
rectangular areas $[1/8,3/8]\times[1/4,3/4]$ and
$[5/8,7/8]\times[1/4,3/4]$, respectively.  TC2 considers $\Source$
formed by $50$ Dirac source points randomly distributed in the region
$\Domain=[0.1,0.9]\times[0.1,0.9]$, while the sink term $\Sink$ is a
single Dirac mass located at $(0.05,0.05)$ with intensity that
balances $\Source$.  TC3 simulates one Dirac source transporting
towards two Dirac sinks, yielding $\Source=\Dirac{(0.5,0.1)}$) and
$\Sink=0.5 \Dirac{(0.4,0.9)} + 0.5 \Dirac{(0.6,0.9)}$.  For this last
case, the exact solution of
\replaced{~\cref{eq:branch-intro}}{
  in the
  sense of~
} is known as a function of the exponent
$\Pbranch$ and will be considered as reference solution. However, as
already noticed in~\cref{sec:pflux-gt-1}, we do not possess the exact
relationship between the exponent $\Pflux$ of our \DMK\ approach and
$\Pbranch$ of the standard \BT\ formulation, and thus only a
qualitative comparison is meaningful.

Sensitivity to initial data is tested for all three TCs by employing
the same $\TdensHIni[i]\ (i=1,2,3)$ used in the previous example.  In
addition, for TC3, we use initial data concentrated along the
reference solution to verify that the dynamics will not move the
solution away from a ``true'' initial guess.  We report the results
for different values of $\Pflux\in\{1.1, 1.5, 2.0, 3.0\}$.  In this
case of ramified transport, we expect $\OptTdensH$ to concentrate on
supports that tend to become progressively singular with respect to
the Lebesgue measure as the mesh is refined. Intuitively, the
numerical transport density should tend towards zero outside these
supports, while it remains positive and may grow indefinitely within
these singular sets. This behavior is magnified as $\Pflux$ grows.
Correspondingly, the ill-conditioning of the linear system
in~\cref{eq:lin-sys-pflux} grows, signaled by the large increase of
the condition number of the system matrix $\Matr{A}$ well beyond the
possibilities of current linear system solvers.  This phenomenon is
intensified as $\Pflux$ increases. For this reason we do not address
values of $\Pflux$ larger than 3.  For the considered simulations,
convergence of the conjugate gradient method is achieved only when
using the ad-hoc preconditioner based on spectral information quickly
described at the end of~\cref{sec:num-met}, developed and thoroughly
tested in~\citep{Martinez-et-al:2017}.

We experimentally test the numerical spatial convergence of the
simulator by solving the same problem on successive refinements of an
initial triangulation $\Triang[\MeshPar]$. For TC1 we use an initial
grid of $1615$ nodes and $3100$ triangles, aligned with the supports
of $\Source$ and $\Sink$, while for TC2 the initial mesh is
characterized by $1661$ nodes and $3192$ triangles.  The $51$ points
where the $\Forcing$ is concentrated coincide with grid nodes.  Again,
for all practical purposes, the value $\TolTime=5\times 10^{-7}$ is
more than enough to reach the equilibrium configuration.  We discuss
the numerical behavior of the model by running simulations with
$\Pflux=1.5$ as a representative example. We look at convergence
towards equilibrium, spatial experimental convergence, behavior of the
\LCF, and sensitivity to initial conditions. For all these tests, the
behavior of the numerical solution and the convergence properties of
the discretization approach are similar also for the other tested
values of $\Pflux$.  In all the numerical simulations, we experimented
strong and sudden variations of $\TdensH$, since the term
$\Delta\TdensH^\tstep=\Matr{B}_{\Pflux}[\UVec[\tstep]]
\left(\TdVec[\tstep]\right)^{\Pflux} -\TdVec[\tstep]$
in~\cref{eq:ee-pflux} can rapidly increase by several orders of
magnitude.  This effect is amplified for larger values of $\Pflux$.
To preserve the stability of the forward Euler scheme we use a time
step $\Deltat$ whose size is tuned according to term
$\Delta\TdensH^\tstep$.

\begin{figure}
  \centerline{
    \includegraphics[width=0.38\textwidth]{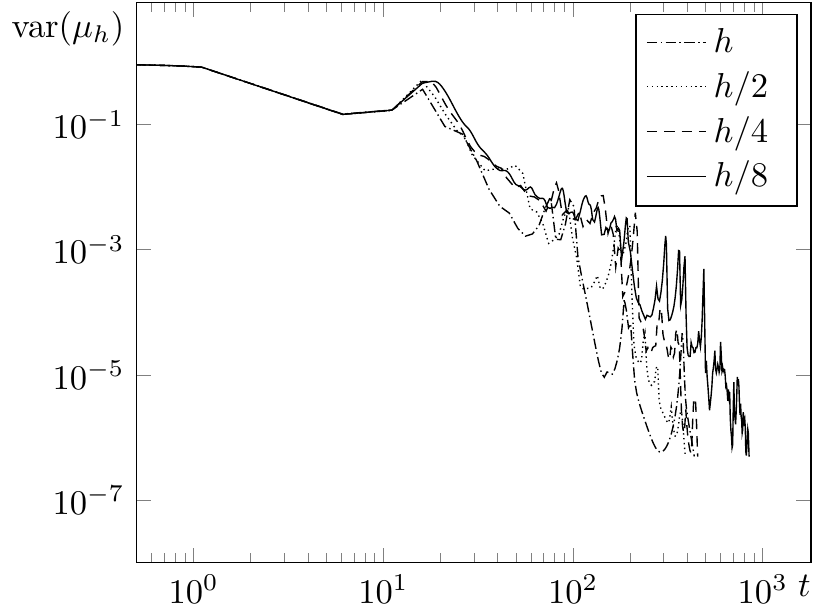}
    \includegraphics[width=0.38\textwidth]{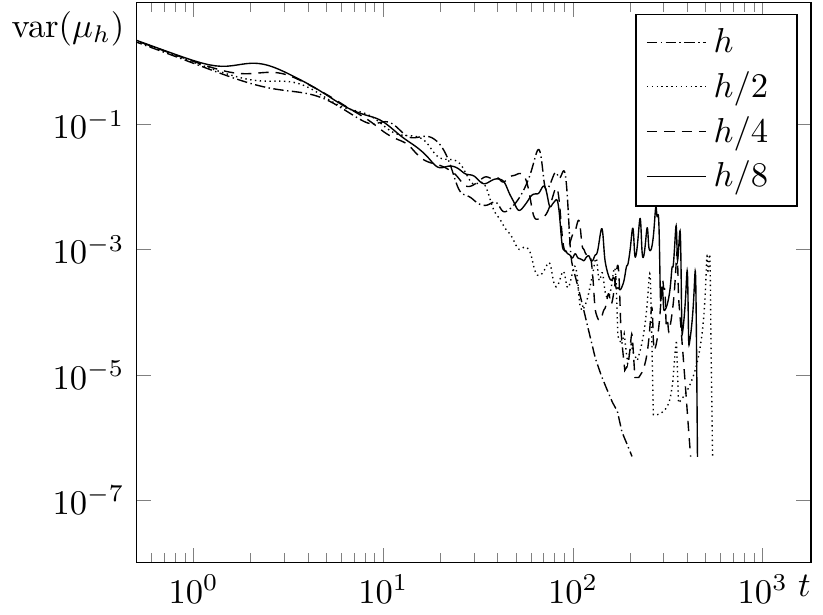}
  }
  \caption{Case $\Pflux=1.5$, $\TdensHIni\equiv 1$.  Time evolution of
    $\Var(\TdensH(t))$ on successive mesh refinements for TC1 (left)
    and TC2 (right).}
  \label{fig:var-150}
\end{figure}

\paragraph{Convergence towards equilibrium}
We first start our discussion by looking at the time evolution of
$\Var(\TdensH(t))$ with uniform initial data $\TdensHIni\equiv 1$ for
both TC1 and TC2 (\cref{fig:var-150}). The behavior is globally
decreasing but not monotone, unlike the case $\Pflux \leq 1$.  After a
reasonably smooth initial transient, the time evolution of
$\Var(\TdensH(t))$ presents oscillations with a frequency that
increases at increasing refinement levels.  Despite this irregular
behavior, in both TCs all simulations seem to converge towards an
equilibrium configuration $\left(\OptTdensH,\OptPotH\right)$ for all
values of $\Pflux$ and for every grid and initial data $\TdensHIni$
considered, thus supporting our conjecture that the proposed
dynamics converge towards a steady state.

\begin{figure}
  \centerline{
    \includegraphics[width=0.8\textwidth]{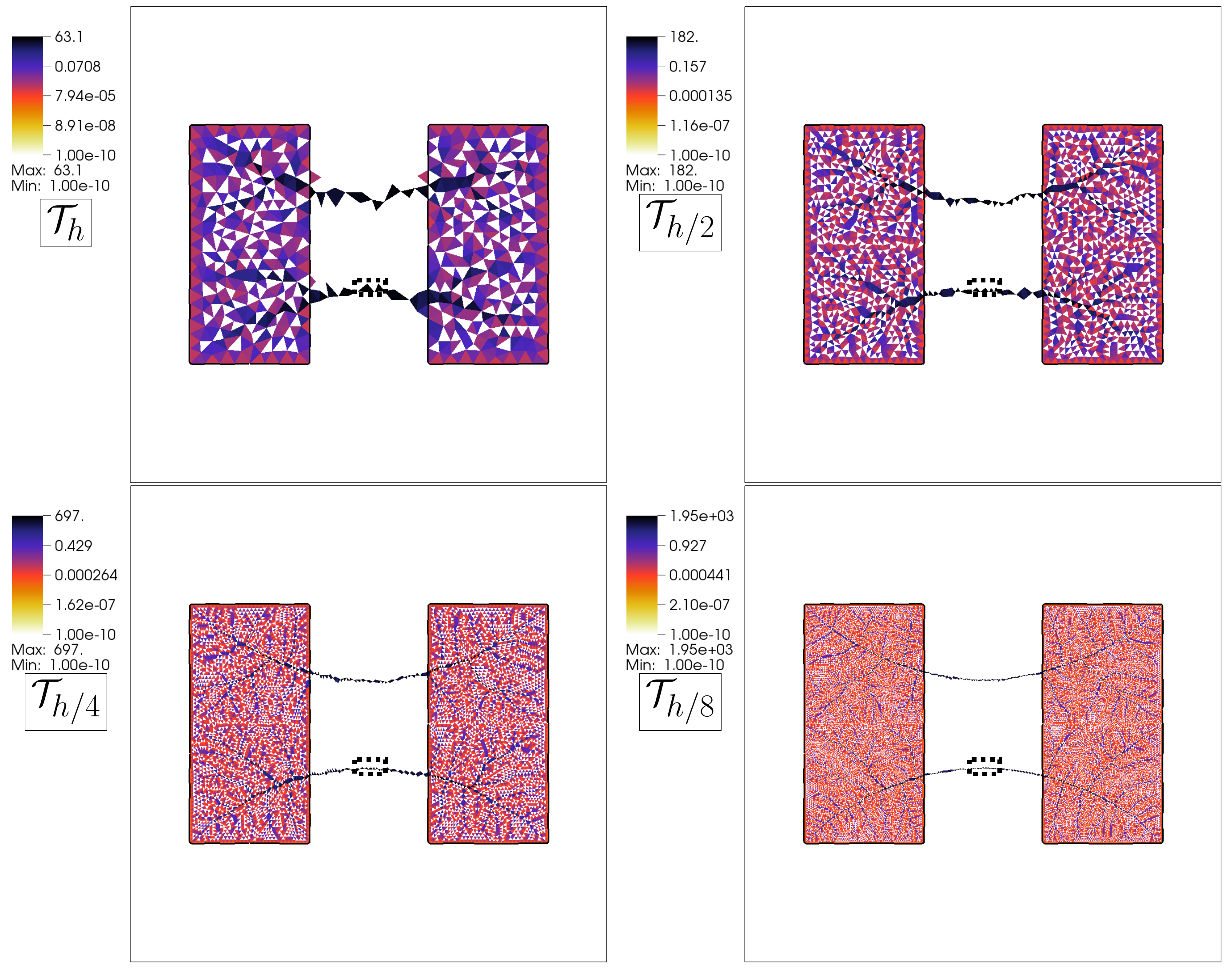}
  }
  \caption{TC1: Numerical approximation $\OptTdensH$ for
    $\Pflux=1.5$ and $\TdensHIni=1$ (logarithmic color scale).  The
    supports of $\Source$ and $\Sink$ are contoured in black.  Four
    successive mesh levels are shown. The dotted rectangle in the lower
    central channel indicate the zoom window displayed
    in~\cref{fig:zoom-150}. The color scale is limited at the minimum
    threshold of $10^{-10}$.}
  \label{fig:rect-150}
\end{figure}

\begin{figure}
  \centerline{
    \includegraphics[width=0.8\textwidth]{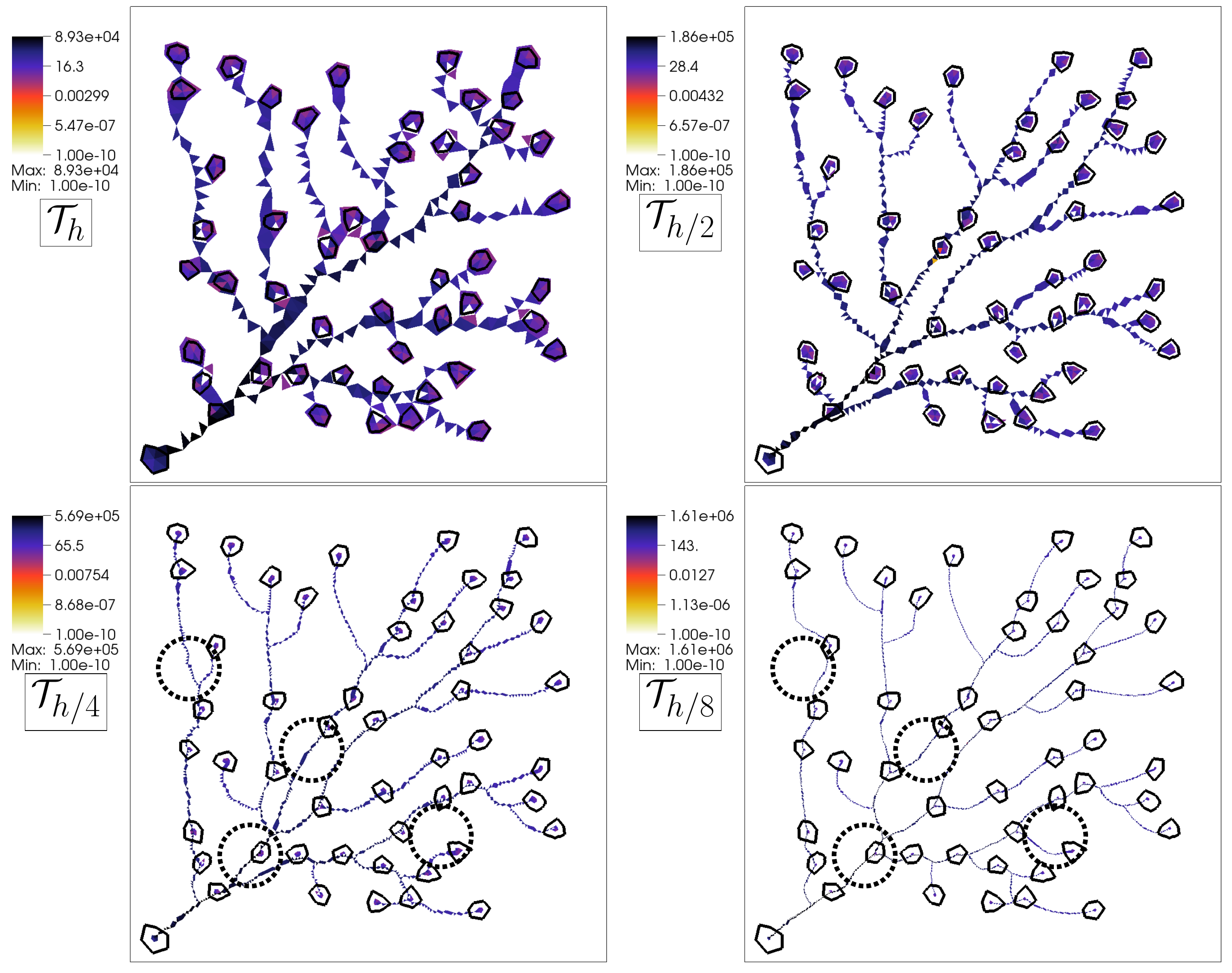}
  }
  \caption{TC2: Numerical approximation $\OptTdensH$ for $\Pflux=1.5$
    (logarithmic color scale).  The initial data $\TdensHIni$ is
    uniformly equal to $1$ on the entire domain.  The small black
    circles indicate the approximate position of the Dirac masses.
    Four successive mesh-refinement levels are shown.  In the bottom
    panels we have indicated with dashed circles the location of
    topological changes between two successively refined network
    structures.  The color scale is limited at the minimum
    threshold of $10^{-10}$.}
  \label{fig:tree-150}
\end{figure}

The spatial distributions of the limit equilibrium configurations
$\OptTdensH$ are shown in~\cref{fig:rect-150,fig:tree-150} for
$\Pflux=1.5$ and $\TdensHIni\equiv 1$ at successive grid refinements.
Looking at the transient (not shown here), we observe supports of
$\TdensH$, defined as the union of all triangles in $\Triang[h]$ where
$\TdensH$ is above a minimal threshold of $10^{-10}$, that initially
coincide with $\Omega$ but then, as time progresses, tend to create
singular structures where $\OptTdensH$ self-organizes in narrow
channels connecting the supports of $\Source$ and $\Sink$.  These
emerging networks present a hierarchical structure in which channels
with higher flow capacity, determined by the values of $\OptTdensH$,
repetitively branch into sub-channels until the whole support of
$\Forcing$ is covered.  The connection between disjoint $\Forcing$
supports is ensured by the formation of a limited number of
concentrated channels.  These emerging structures seem to approach a
singular (one-dimensional) tree-like network, where loops are absent.

Notwithstanding the evident difficulty of comparing $\OptTdensH$ at
different refinement levels, an underlying limit network is clearly
appearing for all tested values of $\Pflux$, grid level, or initial
data $\TdensHIni$.  In particular for TC1, we see
in~\cref{fig:rect-150} that, inside the supports of $\Forcing$,
$\OptTdensH$ forms branching structures with seemingly fractal
features.  In the region outside the supports of $\Forcing$,
$\OptTdensH$ concentrates on a series of connected triangles that
create a tight channel with high conductivity linking the source and
the sink regions.  These effects persists at each refinement level,
and the support of $\OptTdensH$ seems to approximate a one-dimensional
network.  In TC2 we clearly perceive another phenomenon whereby
several branches in the $\OptTdensH$ tree depart from the expected
rectilinear behavior.  We attribute this occurrence to a problem of
mesh-alignment of the numerical solution, presumably to be ascribed to
the extreme spatial irregularity of $\OptTdensH$.  We will discuss
this difficulty later on. However, we are positively surprised by the
capabilities of our numerical scheme to reproduce, albeit with
inaccuracies, these singular structures.  A final remark on the
emerging structure concerns the occurrence of topological changes
between the singular $\OptTdensH$ spatial distributions at different
mesh refinement levels. This is mostly apparent in TC2, where we have
highlighted with dotted circles substantial changes in the network
topology as $\MeshPar$ decreases.  This phenomenon is an indication of
the sensitivity to initial conditions of our \BT\ model, and
correspondingly, to the presence in the \LCF\ of several local minima
to which our dynamics is attracted. This sensitivity will be analyzed
in a later section.

\begin{figure}
  \centerline{
    \includegraphics[trim={2.65cm 25cm 10.5cm 2.1cm},clip,
       width=0.39\textwidth]{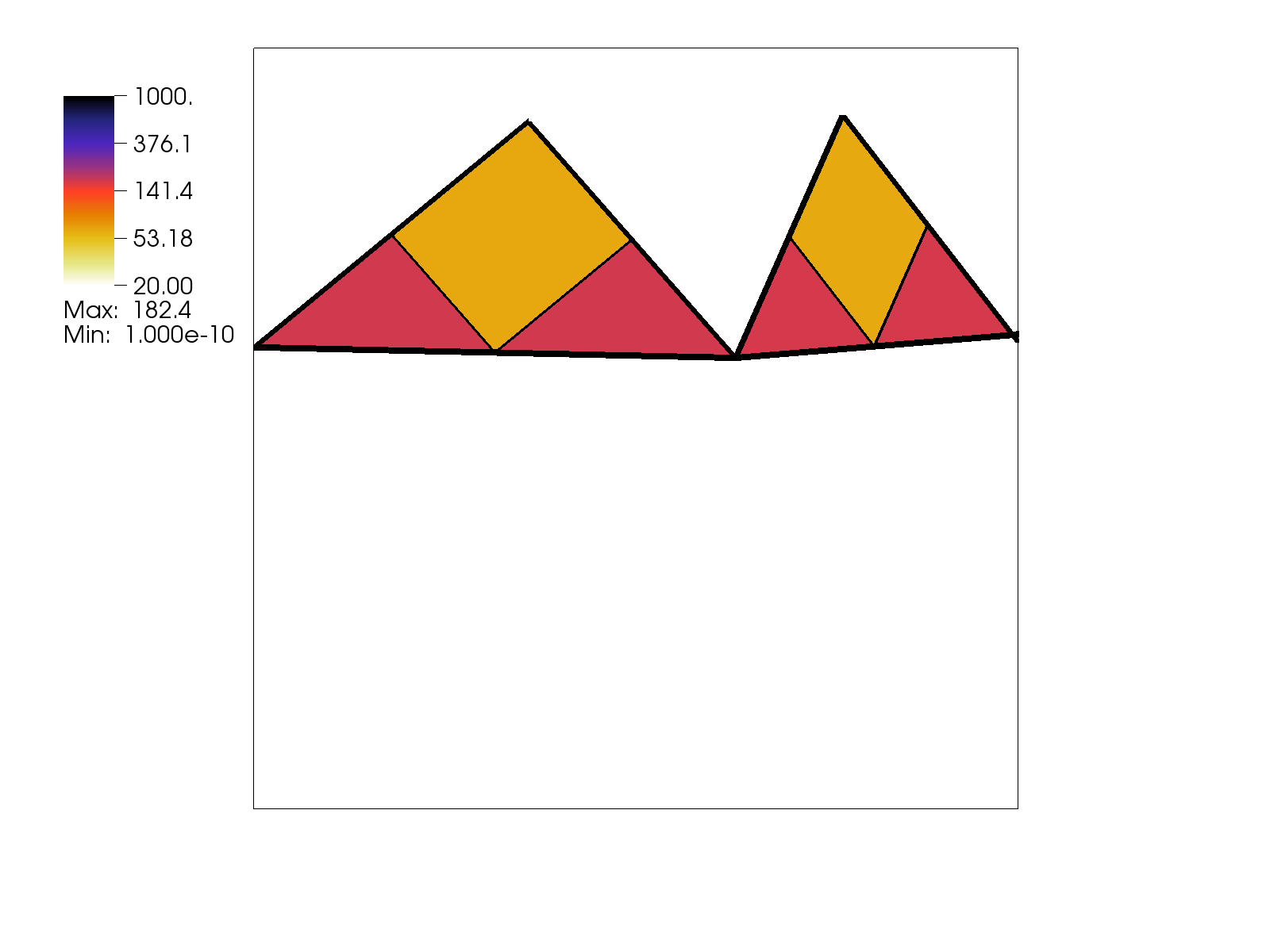}
    \includegraphics[trim={2.65cm 25cm 10.5cm 2.1cm},clip,
       width=0.39\textwidth]{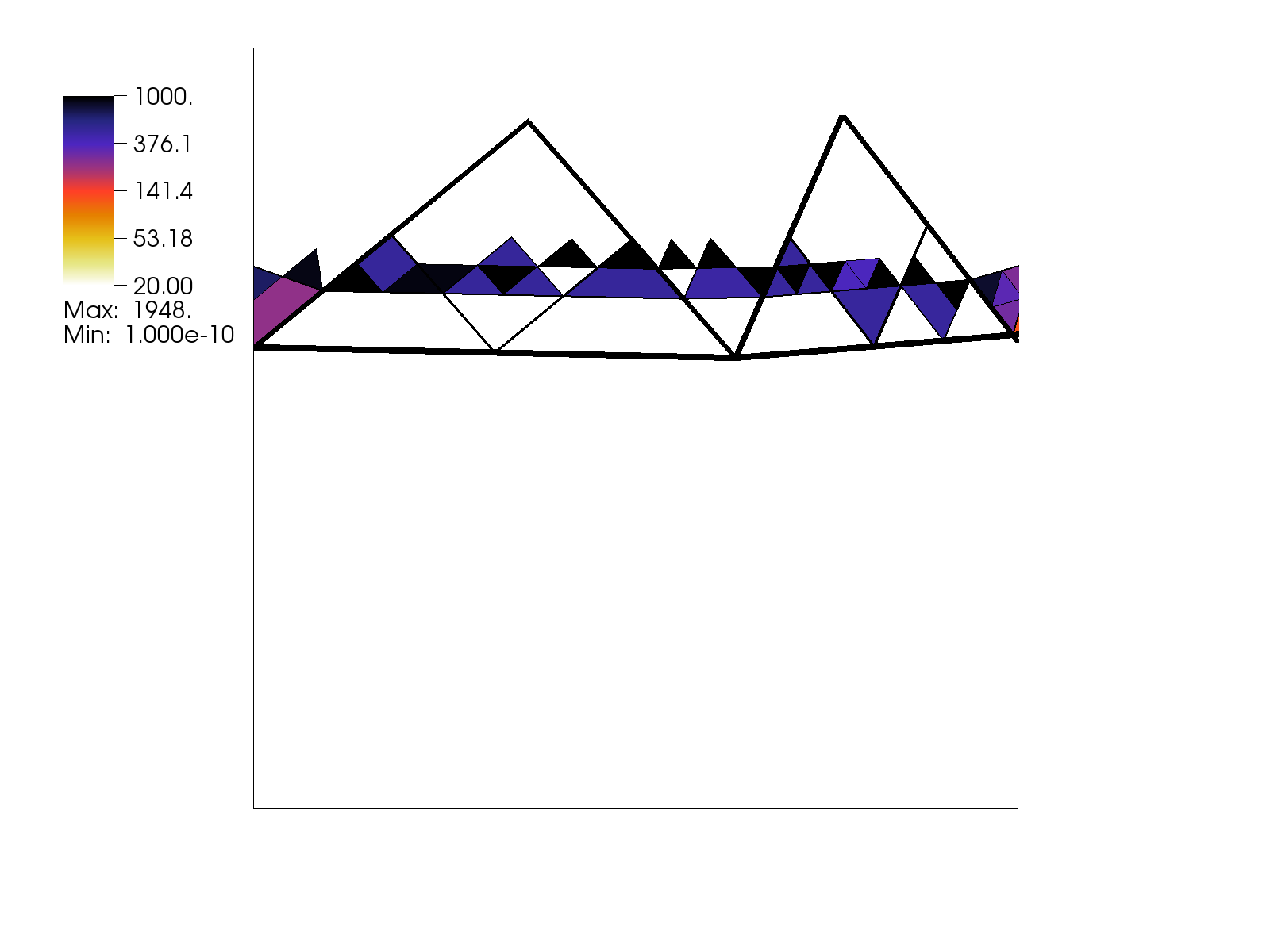}
  }
  \caption{TC1: behavior of $\OptTdensH$ in the zoomed areas located
    in the central channels identified as dotted ovals
    in~\cref{fig:rect-150} (logarithmic color scale).  The left panel
    reports the superposition of $\OptTdensH$ for triangulations
    $\Triang[\MeshPar]$ and $\Triang[\MeshPar/2]$, while the right
    panel superimposes $\Triang[\MeshPar/4]$, and
    $\Triang[\MeshPar/8]$.  Only the triangles where $\OptTdensH$ is
    above the threshold $10^{-10}$ are shown.  }
  \label{fig:zoom-150}
\end{figure}

\paragraph{Experimental convergence of spatial discretization}
Convergence with respect to $\MeshPar$ of these irregular structures
is not easily verified. To better appreciate the differences between
two successive mesh levels we look at a magnification of the solution
of TC1 in the middle of one of the two channels connecting the source
and the sink areas, where $\OptTdensH$ reaches its maximum.  The
zoomed areas are identified in~\cref{fig:rect-150} with dotted ovals.
As seen in~\cref{fig:zoom-150}, this ideally low-dimensional structure
is approximated, at each refinement level, by a sequence of triangles
mostly connected only at nodes.  The width of the ``cross-section'' of
these channels is always formed by one triangle and thus decreases
linearly with $\MeshPar$ as the triangles are refined. The total flux
remains constant since the mass to be transferred from $\Source$ to
$\Sink$ is the same. Correspondingly, $\OptTdensH$ always achieves its
maximum in the triangles forming the central channels, with values that
increase as the mesh is refined.  At finer levels, the spatial
distribution of $\OptTdensH$ becomes more irregular especially within
the supports of $\Source$, as shown in~\cref{fig:rect-150}, displaying
sudden jumps of several orders of magnitude.  Notwithstanding these
irregularities, the solution seems to converge towards some limit
structure, supporting our conjecture that the equilibrium at
$t\Tendsto\infty$ is reached by our dynamics.

As a consequence the condition number of the matrix
$\Matr{A[\TdensH]}$ of the FEM linear system~\cref{eq:lin-sys-pflux}
increases drastically, leading in some extreme cases to
non-convergence of the \PCG\ solver.  As mentioned before, the
spectral preconditioner we use~\citep{Martinez-et-al:2017} is
developed to particularly address these problems.  However, this
strategy requires the construction of the Incomplete Cholesky
factorization with partial fill-in of the \SPD\ $\Matr{A[\TdensH]}$ in
\cref{eq:lin-sys-pflux}, which in radical occurrences of highly
refined meshes and large $\Pflux$ (typically in our experience for
$\Pflux\gg3$) may not exist and cannot be calculated.  In these cases
we abort the simulation.

\begin{figure}
  \centerline{
    \includegraphics[width=0.9\textwidth]{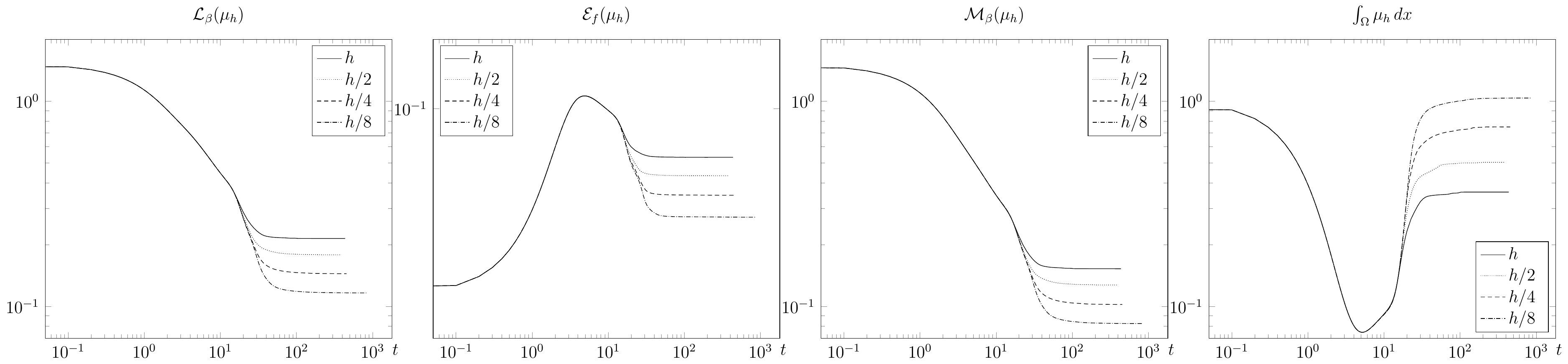}
  }
  \centerline{
    \includegraphics[width=0.892\textwidth]{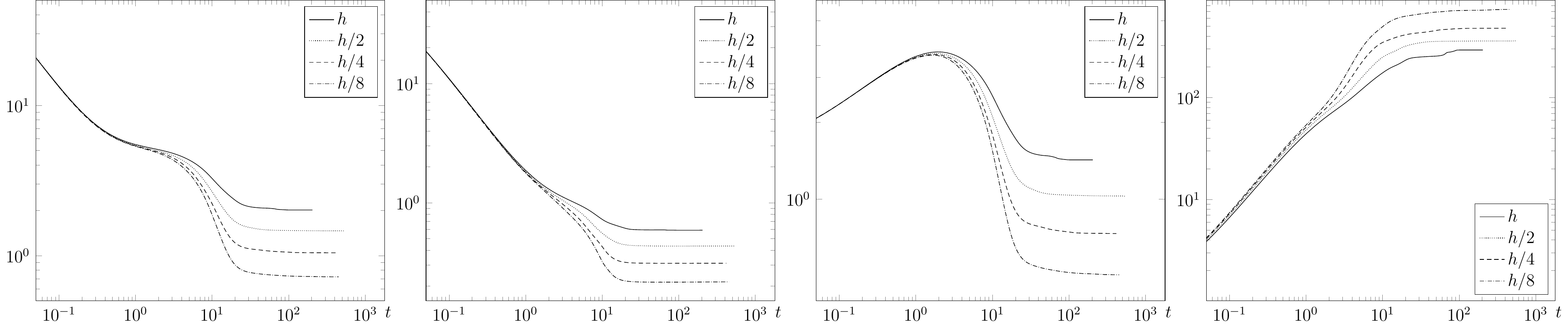}
  }  
  \caption{ Time evolution of $\Lyap_{\Pflux}(\TdensH(t))$,
    $\Ene(\TdensH(t))$, $\Wmass_{\Pflux}(\TdensH(t))$, and
    $\int_{\Domain} \TdensH(t)\dx$ (left to right) for TC1 (top) and
    TC2 (bottom), for all grid refinement levels, in the case
    $\Pflux=1.5$ $\TdensHIni\equiv 1$.  }
  \label{fig:functional-vs-time-150}
\end{figure}

\paragraph{Behavior of the \LCF}
We conjecture that the converged numerical solutions actually
correspond to local minima for $\Lyap_{\Pflux}$.  Thus we look at the
time evolution of the \LCF\ $\Lyap_{\Pflux}(\TdensH(t))$ and its
constituents $\Ene(\TdensH(t))$, $\Wmass_\Pflux(\TdensH(t))$, and
$\int_{\Domain}(\TdensH(t))$.  \Cref{fig:functional-vs-time-150} shows
the time evolution of these components at the different refinement
levels, using $\Pflux=1.5$ and $\TdensHIni\equiv 1$ as starting data.
All the plots show an initial common behavior followed by a distinct
pattern as a function of $\MeshPar$. Intuitively, higher resolutions
allow better exploration of the state space and thus better asymptotic
optimality, possibly leading to varying structures. It is interesting
to note that the behavior for both TC1 and TC2 is consistent inasmuch
as lower $\MeshPar$ leads to lower values of
$\Lyap_{\Pflux}$. Moreover, differences of the equilibrium
$\Lyap_{\Pflux}$-value at consecutive mesh levels remain constant,
corresponding to the constant ratio between subsequent $\MeshPar$
parameters.  These numerical simulations support the statements
in~\cref{prop-decr-lyap-pflux} on the decrease in time of the
$\Lyap_{\Pflux}(\TdensH(t))$. Note that similar results are obtained
for all powers $\Pflux$, initial data $\TdensHIni$, and for both
forcing terms considered.  The presence of local minima is typically
accompanied by sensitivity to initial conditions. Indeed, unlike the
case $\Pflux\le 1$, we observe this dependence, which in our cases
clearly influences the asymptotic value $\Lyap_{\Pflux}(\OptTdensH)$.

\begin{figure}
  \centerline{
    \includegraphics[width=0.3\textwidth]{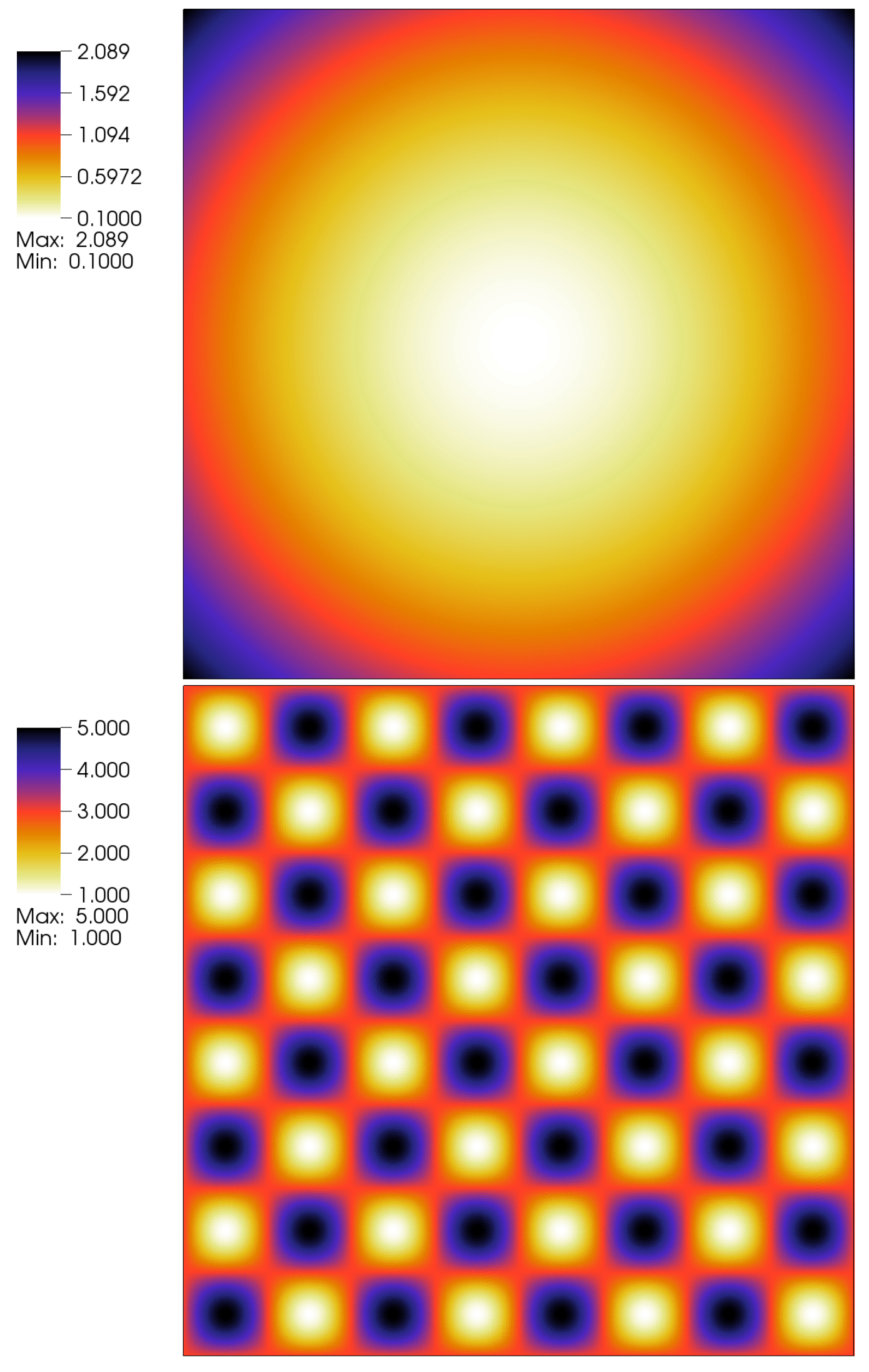}
    \includegraphics[width=0.3\textwidth]{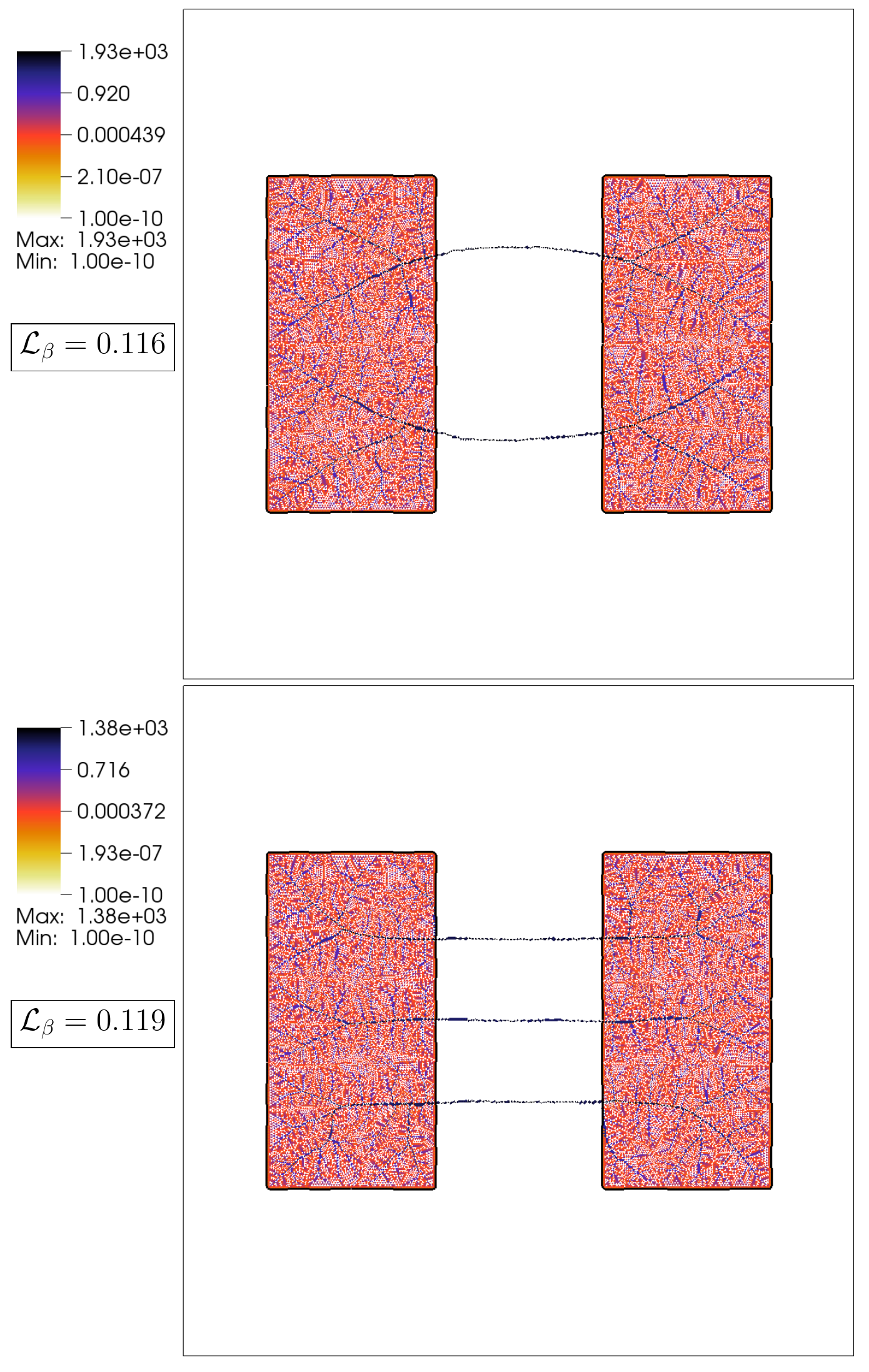}
    \includegraphics[width=0.3\textwidth]{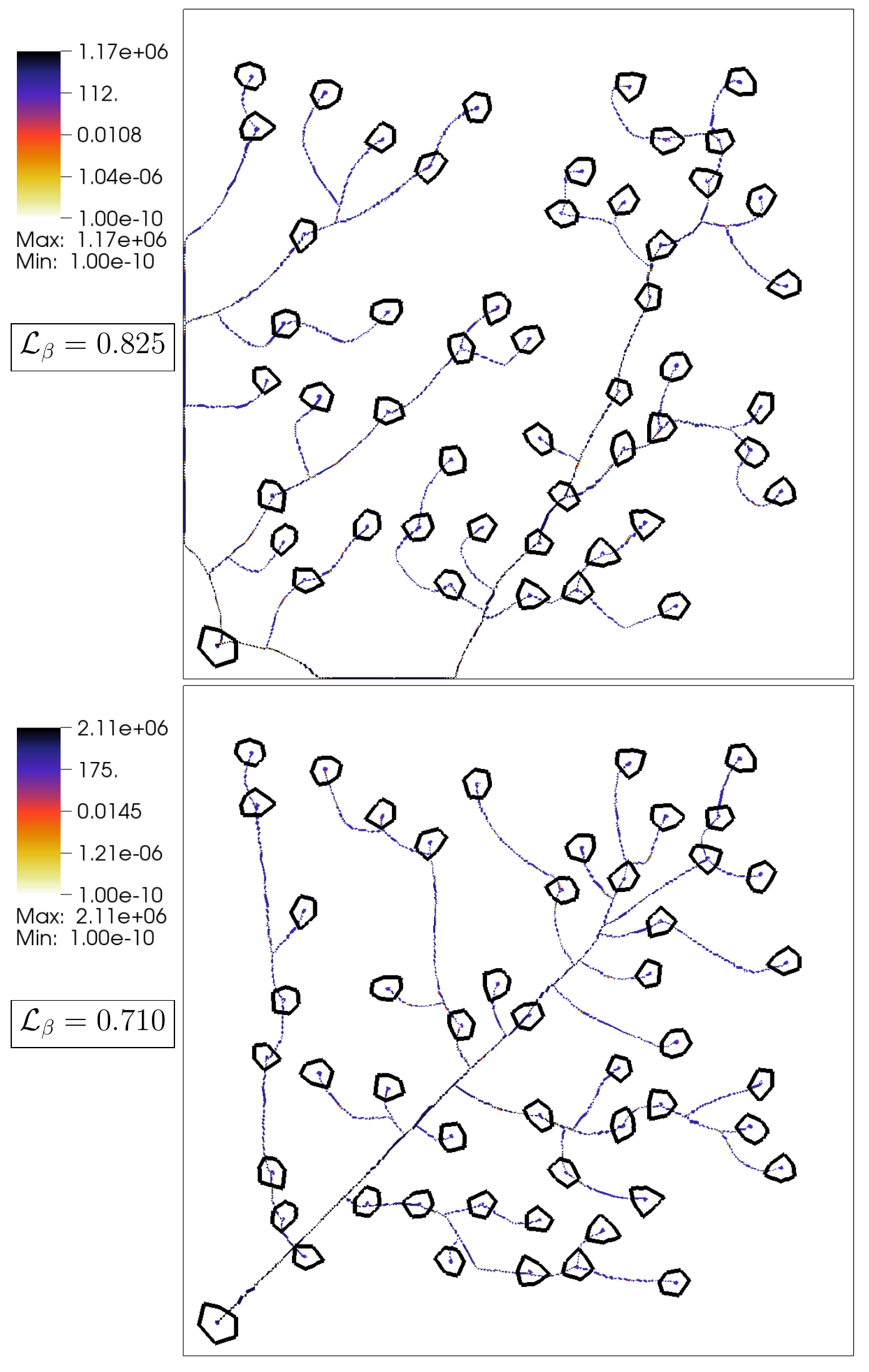}
  }
  \caption{ Spatial distribution of initial data $\TdensHIni[i]$ for
    $i=2,3$ (left panel), and the corresponding asymptotic state
    $\OptTdensH$ for TC1 (center) and TC2 (right) for the finest mesh
    and $\Pflux=1.5$. The value of the \LCF\ $\Lyap_{\Pflux}$ is
    displayed on the box under the color scale. The color scale is
    limited at the minimum threshold of $10^{-10}$.}
  \label{fig:final-pflux-150-td23}
\end{figure}

\begin{table}
  \caption{Values of the \LCF\ $\Lyap_{\Pflux}(\OptTdensH)$ for
    $\Pflux=1.5$ in test cases TC1 and TC2 for different initial
    conditions.}
  \label{tab:lyap-ini}
  \begin{center}
    \begin{tabular}{|l|l|l||l|l|l||l|l|l|}
      \hline
      IC &  TC1 & TC2 &
      IC &  TC1 & TC2 &
      IC &  TC1 & TC2 \\ \hline 
      $\TdensHIni[1]$ & 0.116 & 0.722 & 
      $\TdensHIni[2]$ & 0.116 & 0.825 & 
      $\TdensHIni[3]$ & 0.119 & 0.710 \\
      \hline
    \end{tabular}
  \end{center}
\end{table}

\vspace{1cm}
\paragraph{Sensitivity to initial conditions}
In this paragraph the above-mentioned dependence upon the initial
condition $\TdensHIni$ is explored by running both TC1 and TC2
starting from different spatial patterns. While the case
$\TdensHIni[1]=1$ is shown in the previous sections,
in~\cref{fig:final-pflux-150-td23} we report the $\OptTdensH$ behavior
obtained on the finest grid and for $\Pflux=1.5$ starting from
$\TdensHIni[2]$ and $\TdensHIni[3]$, whose spatial distributions is
shown in the left column. In the case of $\TdensHIni[2]$, where
initial data attain their minimum value of $0.01$ at the center of the
square, the supports of the equilibrium configurations of both TC1 and
TC2 seem to avoid, altogether regions of lower initial density.  On
the other hand, an oscillatory starting configuration leads to added
aggregating and distributive areas in the supports of $\Source$ and
$\Sink$. For TC1 this leads to the formation an additional central
channel connecting $\Source$ and $\Sink$. In TC2, the resulting tree
seems closer to the one obtained with uniform ICs ($\TdensHIni[1]$)
rather \replaced{than}{then} to the case of central
penalization ($\TdensHIni[2]$).  This sensitivity to initial
conditions suggests that the separate network patterns noted
in~\cref{fig:tree-150} are a result of both better resolution and
different time-step size sequences, hinting at the presence of local
minima from which the dynamics is not able to evade.  Indeed, the
value of $\Lyap_{\Pflux}(\OptTdensH)$ (summarized
in~\cref{tab:lyap-ini}) in the case of three central channels (TC3) is
larger than in the other two cases. On the other hand, for TC2,
similar $\Lyap_{\Pflux}$ values are achieved for $\TdensHIni[1]$ and
$\TdensHIni[3]$, with $\TdensHIni[2]$ being higher. Note that the fact
that $\TdensHIni[3]$ results in a smaller value of $\Lyap_{\Pflux}$
with respect to the $\TdensHIni[1]$ solution can be intuitively
justified by observing from the spatial patterns shown
in~\cref{fig:tree-150,fig:final-pflux-150-td23} that the checkerboard
initial condition promotes connections between nodes that are more
straight with respect to case of uniform IC. These results seem to
indicate a rather flat optimization horizon with a number of similar
minima towards which our dynamics leans as a function of the initial
conditions.

\begin{figure}
  \centerline{
    \includegraphics[width=0.8\textwidth]{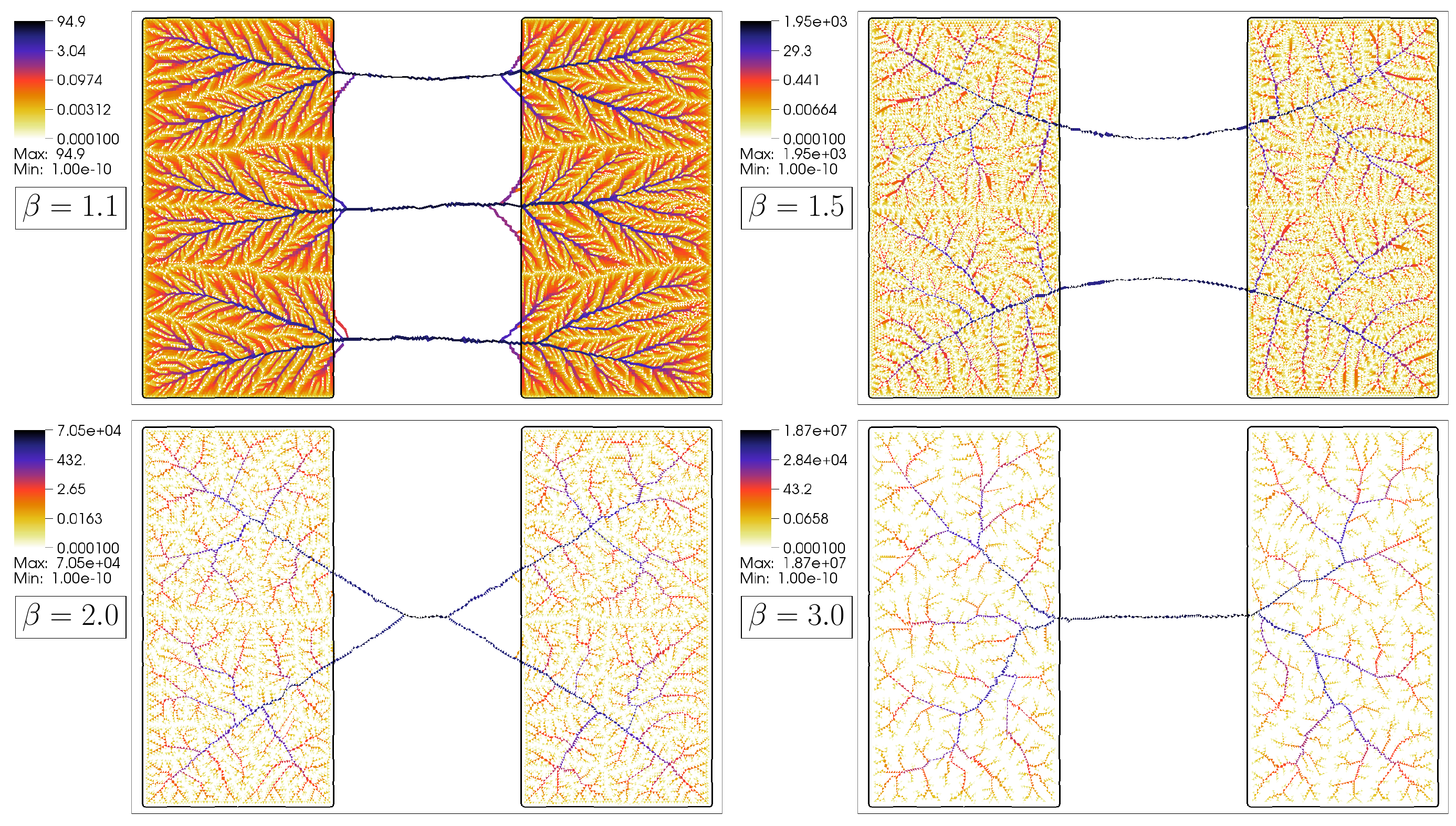}
  }
  \caption{ Behavior of the spatial distribution of
    $\OptTdens_{\Pflux}$ for TC1 for different values of $\Pflux$. The
    color scale starts from $10^{-4}$, but the white regions indicate
    where $\OptTdens_{\Pflux}$ attains the minimum threshold value
    of $10^{-10}$.}
  \label{fig:final-pflux-comparison}
\end{figure}

\begin{figure}
  \centerline{
    \includegraphics[width=0.75\textwidth]{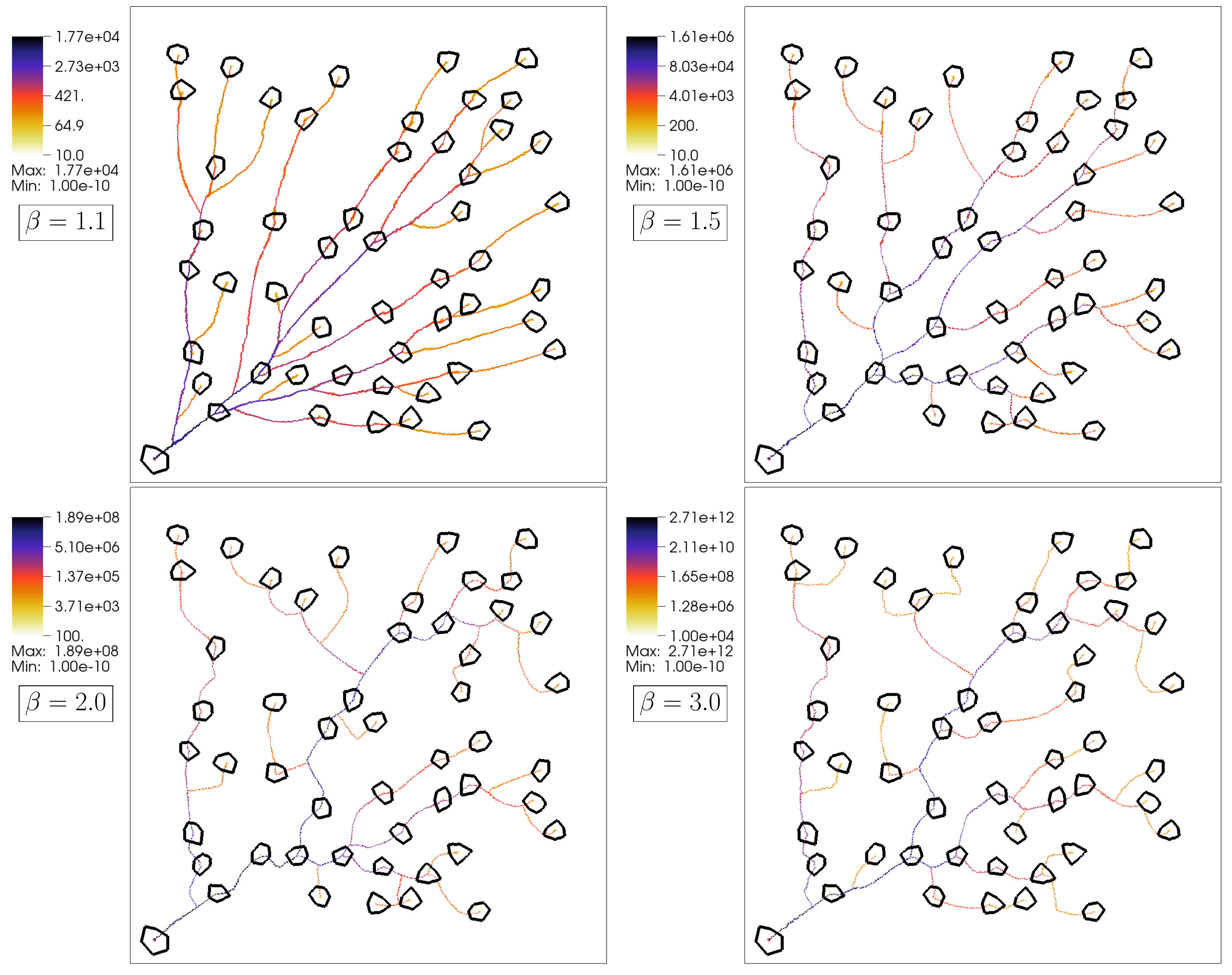}
  }
  \caption{ Behavior of the spatial distribution of
    $\OptTdens_{\Pflux}$ for TC2 for different values of $\Pflux$. The
    color scales start from different initial values, but the white
    regions indicate where $\OptTdens_{\Pflux}$ attains the
     minimum threshold value of $10^{-10}$.}
  \label{fig:dirac-pflux-comparison}
\end{figure}

\paragraph{Sensitivity to $\Pflux$}

This paragraph explores the sensitivity of the proposed model to the
power $\Pflux\in(1.1,1.5,2,3)$.
\Cref{fig:final-pflux-comparison,fig:dirac-pflux-comparison} report
$\OptTdensH$ obtained for the different values of $\Pflux$ on the
finest grid using uniform initial data $\TdensHIni[1]$ for TC1 and
TC2, respectively.  Looking at TC1, for increasing values of $\Pflux$
the proposed dynamics tends to create networks that seem to
increasingly promote concentration. Channels characterized by larger
transport density seem to be sparser and the source and sink areas are
connected by fewer singular structures.  In fact, the number of
central channels created at the equilibrium varies from three for
$\Pflux=1.1$ to one for $\Pflux=3$, with the final configuration
showing no branching points in the central channel.  In the TC2 case,
a number of topological changes in the emerging networks are clearly
noticed in the sequence.  Moreover, inaccuracies emerge in the form of
curve-shaped connections between branching points approximating the
expected straight lines.  These inaccuracies grow as $\Pflux$
increases and are clearly visible already for $\Pflux=2.0$.  We argue
that they are caused by the combined effects of grid alignment and the
dependence upon the initial data, leading to a configuration possibly
related to a local minimum of the \LCF.

For both test cases we note branching angles that increase with
$\Pflux$, as observed in the \BT\ theory of~\citet{Xia:2003} for
decreasing values of $\Pbranch$.  Analogously, for TC2 the branching
points move progressively away from the source nodes at increasing
$\Pflux$. These results suggest that $\Pflux$ must increase as
$\Pbranch$ decreases.

\begin{figure}
  \centerline{
      \includegraphics[width=0.6\textwidth]
      {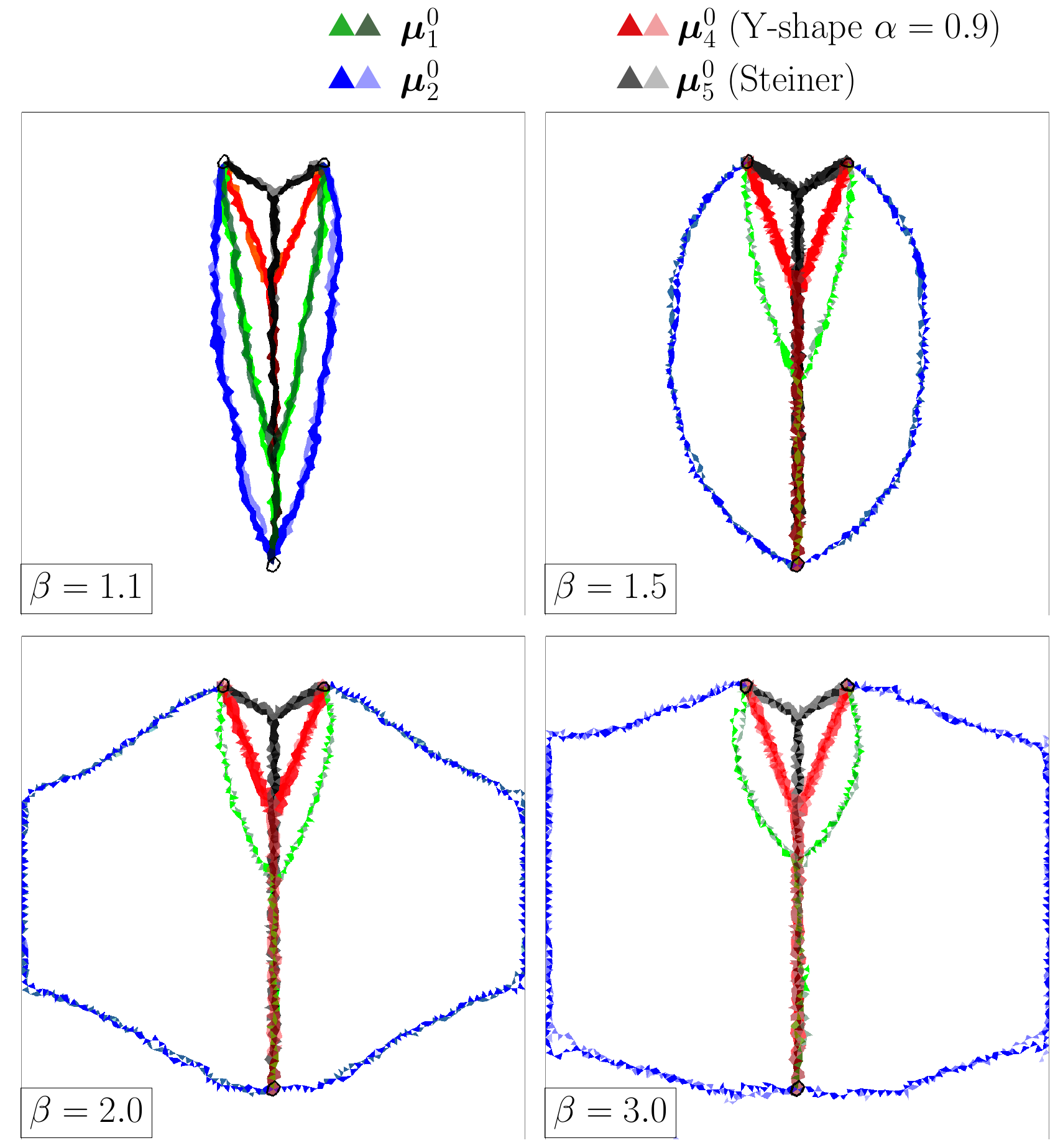}
  }
  \caption{ Spatial distribution of $\OptTdensH$ for TC3 starting form
    different initial data, as identify by different colors. The green
    and blue curves are obtained using $\TdensHIni[1]$ and
    $\TdensHIni[2]$, respectively. The \deleted{brown,} red and the
    black curves report the solutions obtained by using an initial
    distribution concentrated along the Y-shaped known solution of the
    standard branched transport problem for $\Pbranch=0.9, 0$.  The
    different panels show the equilibrium configurations of $\TdensH$
    for $\Pflux=1.1,1.5,2.0,3.0$.  Additionally, each color is
    represented with two different intensities to show the solutions
    obtained with two different meshes of the same size but different
    node distribution. The color scale is limited at the minimum
    threshold of $10^{-10}$.}
  \label{fig:y-comparison}
\end{figure}
\paragraph{Test Case 3}

TC3 is designed to address two fundamental questions arising from the
results of TC1 and TC2. First, we would like to explore the influence
on $\OptTdensH$ of the grid geometry and of the initial data
$\TdensIni$.  Then we would like to acquire some intuition on a
possible relationship $\Pflux(\Pbranch)$, in addition to the above
observation that it must be a decreasing function.  In this test case
the optimal vector field $\OptVel$ solution
of~\cref{eq:min-divergence} is supported on a Y-shaped graph that
branches at a point with coordinates $(0.5,c(\Pbranch))$ with
$c(\Pbranch)$ a function that grows from $0.1$ for $\Pbranch=1.0$ to
$0.8422$ for $\Pbranch=0.0$~\citep{Xia:2015}.  Note that the latter
value of $\Pbranch$ corresponds to Steiner problem, for which the path
branches at constant angles equal to $2/3\pi$.
\Cref{fig:y-comparison} shows a number of simulation results for
different values of $\Pflux$, different initial conditions, and
different meshes. Each curve in a panel displays the large-time
solution obtained with initial condition specified by the color codes
in the top legend. In particular, the green and blue curves correspond
to $\TdensHIni[1]$ and $\TdensHIni[2]$, respectively.  The initial
conditions identified with $\TdensHIni[4]$ and $\TdensHIni[5]$, (red
and gray, respectively) are the projection on $\Triang[\MeshPar]$ of
the \BT\ Y-shaped transport path for $\Pbranch=0.9$ and $\Pbranch=0$.
For each initial guess, we show two numerical solutions obtained from
simulations run with two different triangulations of the same size but
varying nodal distribution, identified with the same color but
different intensity.  The results reported in~\cref{fig:y-comparison}
suggests that the choice of the initial data $\TdensIni$ has a much
stronger influence on the steady state configuration $\OptTdensH$,
than the triangulation $\Triang[\MeshPar]$ used in the discretization
of ~\cref{eq:sys-pflux}.  In fact, we note that, independently on the
initial data $\TdensIni$ or the power $\Pflux$, the supports of
$\OptTdensH$ reported in~\cref{fig:y-comparison} with darker and
lighter colors clearly concentrate on the same limit structure, which
is only marginally influenced by the topological constraint imposed
a-priori by the graph associated with the triangulation.

On the other hand, we clearly notice that the supports of $\OptTdensH$
tend to concentrate on regions where $\TdensIni$ is larger, a behavior
already observed in TC1 and TC2.  For
$\TdensIni=\TdensHIni[4], \TdensHIni[5]$, the dynamics is not able to
drive $\TdensH$ away from the strongly pre-imposed paths.  For
$\TdensIni=\TdensHIni[2]$, $\OptTdensH$ distributes on two separate
branches concentrating on regions with initially higher
conductivities.  Such behavior is more pronounced as $\Pflux$ grows.
For the case $\TdensIni=\TdensHIni[1]$, where the initial data is
uniformly distributed and there is no a-priori bias, the support of
$\OptTdensH$ forms a Y-like shape, with a bifurcation point at
$(0,c(\Pflux))$ with $c(\Pflux)$ that is gradually increasing with
$\Pflux$. However, we note how the point $(0,c(\Pflux))$ does not
quite reach the bifurcation point $(0,c(\Pbranch))$ of the reference
solution for $\Pbranch=0.9$ even for the largest value $\Pflux=3.0$.

Unfortunately, the only conclusion that can be deduced from this and
the other test cases not reported here is the already mentioned
decreasing behavior of $\Pflux(\Pbranch)$.  The strong dependence of
our dynamics on the initial data does not allow a more accurate
characterization of this relation.

\section{Conclusions and discussion}
We have presented and discussed an extension of the \DMK\ model where
the transient equation for $\Tdens$ is modulated by a power law of the
transport flux with exponent $\Pflux>0$. The original \DMK\ model
introduced in~\cite{Facca-et-al:2018} is a subset of the version
discussed in this paper when $\Pflux=1$.

We conjecture that, for $0<\Pflux<1$, the long-time limit of the
extended \DMK\ model is equivalent to a $\Plapl$-Poisson Equation
for $\Plapl=(2-\Pflux)/(1-\Pflux)$, and, consequently, it is
equivalent to the \CTP.  Theoretical and numerical evidence support
our claims and show that the extended \DMK\ represents a new dynamic
formulation of the $\Plapl$-Poisson equation and can be proposed as a
relaxation for the efficient numerical solution of $\Plapl$-Laplacian.

In the case $\Pflux>1$, we claim a link between the steady state
solution $(\OptTdens,\OptPot)$ of our extended \DMK\ model and
solutions of \BTP s.  In this case, the complex solutions emerging
from our dynamics remind of singular distributions typical of branched
transport problems. However, when comparing both theoretically and
numerically our \BTP\ with the more classical \BTP\ studied, e.g.,
in~\citet{Xia:2003,Santambrogio:2010,
  Oudet-Santambrogio:2011,Xia:2014,Xia:2015,Pegon:2018} differences
arise.  In primis we need to acknowledge that our formulation, being
based on the FEM approach, is based on densities that can be
approximated with the Lebesgue measure.  Next, we are not able to
formulate an exact relationship between the equilibrium configurations
$(\OptTdens,\OptPot)$ and the more classical solutions of \BTP.  On
the other hand, several numerical examples strengthen our confidence
that our approach leads to interesting solutions that are promising in
the quest for numerical solutions of \BTP s.  Indeed, our approach
seems to be efficient and robust enough to produce trusty results at
least for distributed sources.

\section*{Acknowledgments}
This work was partially funded by the the UniPD-SID-2016 project
“Approximation and discretization of PDEs on Manifolds for
Environmental Modeling” and by the EU-H2020 project
``GEOEssential-Essential Variables workflows for resource efficiency
and environmental management'', project of ``The European Network for
Observing our Changing Planet (ERA-PLANET)'', GA 689443.

\bibliographystyle{abbrvnat}
\bibliography{strings,biblio}
\end{document}